\documentclass[11pt,reqno]{amsart}
\usepackage[T1]{fontenc}
\usepackage[utf8]{inputenc}
\usepackage{amsmath,amssymb,amsthm,mathtools,mathrsfs}
\usepackage{microtype}
\usepackage{graphicx,flafter}
\usepackage[margin=1.15in]{geometry}
\usepackage[colorlinks=true,linkcolor=blue,citecolor=blue,urlcolor=blue]{hyperref}
\hypersetup{pdftitle={The injectivity threshold of a Gaussian ReLU layer},
pdfauthor={Xiaohui Xie}}
\numberwithin{equation}{section}
\newtheorem{theorem}{Theorem}[section]
\newtheorem{proposition}[theorem]{Proposition}
\newtheorem{lemma}[theorem]{Lemma}
\newtheorem{corollary}[theorem]{Corollary}
\theoremstyle{definition}
\newtheorem{definition}[theorem]{Definition}
\theoremstyle{remark}

\newcommand{\R}{\mathbb R}

\newcommand{\E}{\mathbb E}
\newcommand{\Prob}{\mathbb P}
\newcommand{\ind}{\mathbf 1}
\newcommand{\dd}{\mathrm d}
\newcommand{\norm}[1]{\lVert #1\rVert}
\newcommand{\abs}[1]{\lvert #1\rvert}

\DeclareMathOperator{\Var}{Var}

\DeclareMathOperator{\ReLU}{ReLU}

\title{The injectivity threshold of a Gaussian ReLU layer}
\author{Xiaohui Xie}
\address{University of California, Irvine}
\email{xhx@uci.edu}
\date{}
\begin{document}
\begin{abstract}
We determine how wide a randomly initialized ReLU layer must be to
distinguish every pair of inputs. For independent Gaussian weights and
zero bias, as the dimensions grow at a fixed output-to-input ratio,
the probability of global injectivity tends to zero below a sharp
threshold and to one above it. We characterize this threshold exactly
through a variational formula for the limiting minimum number of active
neurons per input coordinate. The formula accounts for correlations
among input directions at arbitrarily many levels. We prove this formula
rigorously, establishing a prediction from the statistical physics of the
spherical perceptron. The proof compares systems of nearby dimensions and controls a
Gaussian interpolation by calibrating responses to auxiliary fields.
A uniform volume estimate connects the soft minimum to the worst input
direction. An exact-arithmetic calculation places the
threshold below seven output neurons per input coordinate. Numerical
evaluations suggest a value near $6.698$.
\end{abstract}
\maketitle
\raggedbottom
\section{Introduction}

A neural representation preserves all information about its input only if
distinct inputs have distinct outputs. This property, called
\emph{global injectivity}, is needed to recover signals from learned
representations or latent variables from generative models
\cite{Puthawala2022Globally}. ReLU sets negative preactivations to zero,
and which neurons remain active depends on the input. We ask how wide a
Gaussian ReLU layer must be to distinguish every pair of inputs.

Let $n$ and $m$ be the input and output dimensions, respectively, and let
$W\in\R^{m\times n}$ have independent standard normal $N(0,1)$ entries.
We study the zero-bias map
\[
 f_W:\R^n\longrightarrow\R^m,
 \qquad f_W(x)=\ReLU(Wx/\sqrt n),
\]
where $\ReLU(t)=\max\{t,0\}$ acts coordinatewise. Write
$p_{m,n}=\Prob\{f_W\text{ is injective}\}$ for the probability over the
weights and $\alpha=m/n$ for the output-to-input ratio.
A Gaussian linear map has full column rank almost surely when $m\geq n$.
After applying ReLU, the relevant rows are those whose outputs are
positive. At a fixed nonzero input, each neuron is independently active
with probability $1/2$. Injectivity requires enough active rows in every
direction, including directions selected after observing the weights.

Let $w_i\in\R^n$ be the transpose of the $i$th row of $W$, and let
$S^{n-1}=\{u\in\R^n:\|u\|_2=1\}$ be the Euclidean unit sphere. The
active-neuron count in direction $u$, with $\ind$ denoting an indicator
and $\cdot$ the Euclidean inner product, is
\begin{equation}\label{eq:energy}
 H_W(u)=\sum_{i=1}^m\ind\{w_i\cdot u>0\},\qquad u\in S^{n-1}.
\end{equation}
Almost surely,
\[
 f_W\text{ is injective}\quad\Longleftrightarrow\quad
 \min_{u\in S^{n-1}}H_W(u)\geq n.
\]
To see the idea, fix an open region on which all row signs are constant.
The layer is linear there, and its active rows determine its rank.
Fewer than $n$ active rows leave a direction in which two nearby inputs
have the same output. General position supplies the converse; directions
on row hyperplanes can also be reduced to open regions
(Lemma~\ref{geo:criterion}). The threshold is therefore determined by the
input direction with the fewest active neurons.

\subsection{Main results}
We approach the minimum through an integral version of the soft minimum
used in optimization. Let $\mu_n$ be uniform probability measure on
$S^{n-1}$, and for $\beta>0$ define
\begin{equation}\label{eq:partition}
 Z_{n,m}(\beta)=\int_{S^{n-1}}e^{-\beta H_W(u)}\,\mu_n(\dd u),
 \qquad \Phi_{n,m}(\beta)=\frac1n\E\log Z_{n,m}(\beta).
\end{equation}
Here $\E$ averages over $W$. The quantity $-\log Z_{n,m}(\beta)/\beta$ is
the soft minimum. Increasing $\beta$ puts more weight on directions with
few active neurons. In statistical physics, $\beta$ is the inverse
temperature, $Z_{n,m}$ is the partition function, and $\Phi_{n,m}$ is the
pressure, or normalized expected log partition function. The expectation
of the soft minimum per input coordinate is $-\Phi_{n,m}/\beta$.

Section~\ref{sec:variational} defines an explicit variational function
$P_*(\alpha,\beta)$ through recursive one-dimensional Gaussian integrals.
Its trial parameters describe hierarchies of correlations between input
directions. Optimizing over all finite hierarchies gives the full replica
symmetry breaking variational formula. This function determines both the
limiting pressure and the injectivity threshold.

\begin{theorem}[Injectivity threshold]\label{main:threshold}
For every finite $\alpha\geq0$, the limit
\begin{equation}\label{eq:limiting-energy}
 e(\alpha)=\lim_{\beta\to\infty}\frac{-P_*(\alpha,\beta)}{\beta}
\end{equation}
exists. For every deterministic sequence $m_n/n\to\alpha$,
\begin{equation}\label{eq:ground-energy-limit}
 \frac1n\min_{u\in S^{n-1}}H_W(u)\longrightarrow e(\alpha)
 \quad\text{in probability and in }L^1.
\end{equation}
There is a unique $\alpha_\star\in(0,\infty)$ such that
$e(\alpha_\star)=1$. For every such sequence with $\alpha\ne\alpha_\star$,
\[
 \lim_{n\to\infty}p_{m_n,n}=
 \begin{cases}
 0,&\alpha<\alpha_\star,\\
 1,&\alpha>\alpha_\star.
 \end{cases}
\]
\end{theorem}

The function $e(\alpha)$ is the limiting minimum number of active neurons
per input coordinate. Injectivity changes phase when this quantity
crosses one.
The $L^1$ convergence in \eqref{eq:ground-energy-limit} means convergence
in expected absolute error. The theorem leaves the behavior exactly at
$\alpha=\alpha_\star$ open.

Elementary geometric bounds and an explicit variational trial give
$4\leq\alpha_\star<7$. The upper bound is certified by exact arithmetic
in Section~\ref{sec:numerics}. Numerical evaluations with increasingly
many correlation levels give values near $6.698$, consistent with the
prediction of \cite{Maillard2025Injectivity}; these decimals are not a
certified enclosure. Figure~\ref{fig:threshold-overview} summarizes the
phase transition and the status of these bounds.

\begin{figure}[htbp]
 \centering
 \includegraphics[width=\textwidth]{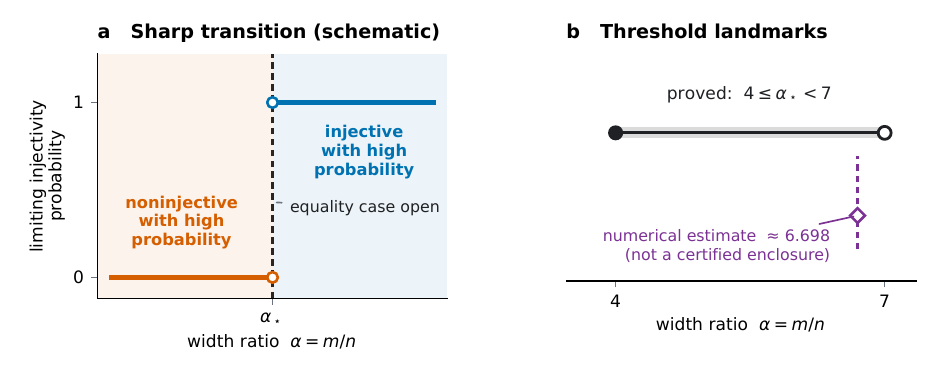}
 \caption{The Gaussian ReLU injectivity threshold. The phase transition
 is schematic: for each fixed ratio below $\alpha_\star$, the probability
 of global injectivity tends to zero; above it, the probability tends to
 one. The exact threshold is characterized by $e(\alpha_\star)=1$.
 The bounds $4\leq\alpha_\star<7$ are rigorous, while the location near
 $6.698$ comes from numerical variational trials. The figure specifies
 neither the probability at equality nor a finite-dimensional transition
 curve.}
 \label{fig:threshold-overview}
\end{figure}

These widths concern independent Gaussian weights. Structured weights can
achieve injectivity with only $2n$ outputs: the map
$x\mapsto(\ReLU(x),\ReLU(-x))$ recovers $x$ by subtracting its two output
blocks. Such constructions and their stability are studied in
\cite{Puthawala2022Globally}. The Gaussian threshold quantifies the extra
width needed when the rows are chosen independently at random.

\begin{samepage}
\begin{theorem}[Limiting pressure]\label{main:pressure}
For every $\alpha\in[0,\infty)$, every $\beta>0$, and every deterministic
sequence of nonnegative integers $m_n$ with $m_n/n\to\alpha$,
\begin{equation}\label{eq:pressure-limit}
 \lim_{n\to\infty}\Phi_{n,m_n}(\beta)=P_*(\alpha,\beta).
\end{equation}
For $P_n\sim\operatorname{Pois}(\alpha n)$ independent of the Gaussian rows,
\[
 \lim_{n\to\infty}\frac1n\E\log Z_{n,P_n}(\beta)=P_*(\alpha,\beta),
\]
where the expectation includes both $P_n$ and the Gaussian rows.
\end{theorem}
\end{samepage}

The formula holds at every finite width ratio and every $\beta>0$.
To recover the true minimum from the soft minimum, we must also show that
directions with nearly minimal counts occupy enough of the sphere to
contribute to the integral, uniformly in $n$. Section~\ref{sec:strategy}
explains this step and the two comparisons that prove the pressure formula.

\subsection{Related work}
\emph{Geometry of ReLU layers.}
Puthawala et al.~\cite{Puthawala2022Globally} characterized injectivity
through the spanning properties of active rows, gave bounds for Gaussian
layers, and studied the stability of their inverses. Their work also
covers structured fully connected and convolutional layers and networks.
Our result identifies the asymptotic width threshold in the single-layer,
independent Gaussian model. The classical hyperplane and separation
formulas of Cover~\cite{Cover1965Geometrical} and
Wendel~\cite{Wendel1962Problem} provide useful sign-pattern bounds, but
identifying the minimum active count requires more than counting feasible
patterns.

\emph{The perceptron formulation.}
Maillard, Bandeira, Belius, Dokmani\'c and Nakajima
\cite{Maillard2025Injectivity} formulated the Gaussian problem as a
spherical perceptron: its parameter is a unit vector, and its energy counts
violated Gaussian linear inequalities. They proposed the full replica
symmetry breaking formula used here, with a predicted threshold near
$6.698$. They proved
noninjectivity when the limiting
zero-temperature energy is below one \cite[Corollary 1.5]{Maillard2025Injectivity}
and a replica-symmetric sufficient bound near $7.6477$
\cite[Theorem 1.8]{Maillard2025Injectivity}.

Theorem~\ref{main:pressure} proves their Conjecture 1.7(i), including
existence of the dimension limit at every finite density and positive
temperature. Theorem~\ref{main:threshold} proves the ground-energy identity
in their equation (10) and both implications of Conjecture 1.6. Together
these establish the variational sharp-threshold characterization underlying
their Result 1.1. The numerical interval in that result remains
uncertified here. Their Conjecture 1.7(ii), which identifies a minimizing
profile with the overlap distribution of the unperturbed Gibbs measure,
is a separate assertion and is not proved in this paper.

\emph{Lifted Gaussian comparisons.}
Stojnic~\cite{Stojnic2024Injectivity} treats the injectivity obstruction
as a Gaussian min--max feasibility problem. His Tables 2--4 report
$7.4486$ at the partial second level and $6.7157$, $6.7004$ at the full
second and third levels. Partial lifting has a one-sided comparison
interpretation, giving sufficient conditions for injectivity
\cite[Corollary 4]{Stojnic2023Duality}. The full-lifting equalities in
\cite[Theorems 1--2]{Stojnic2024Injectivity} use the stationarized
framework of \cite[Theorem 3]{Stojnic2023Stationarization}. That framework
assumes stationary parameter paths together with overlap concentration
along the interpolation, or vanishing interpolation remainders.
Endpoint stationary equations alone do not establish these hypotheses.

Our proof concerns the positive-row-count partition function
\eqref{eq:partition}; the auxiliary bilinear partition in
\cite[equations (14)--(18)]{Stojnic2024Injectivity} is a different
finite-temperature object. For the row-count model, we establish the
needed interpolation sign from joint overlap identities and response calibration at deterministic
covariance maxima, without assuming a stationary parameter path.

\emph{Cavity and interpolation methods.}
The dimension comparison is related to the cavity variational principle
of Aizenman, Sims and Starr~\cite{Aizenman2003Extended}, while the upper
bound follows the Gaussian interpolation approach of
Guerra~\cite{Guerra2003Broken}. Probability cascades encode the trial
hierarchies \cite{Arguin2007RPC}; the representation and ultrametricity
results for overlap arrays \cite{Panchenko2014SK,Panchenko2013Ultrametricity}
identify the limiting objects used in the comparison. The row cost jumps
whenever a row field crosses zero, and this discontinuity requires additional arguments. Restriction to a random
subspace gives the projection cavity comparison between dimensions.
Response calibration controls the error in the upper-bound interpolation.
A uniform volume estimate then connects the soft minimum to the fewest
active neurons in any direction.
Section~\ref{sec:strategy} explains these steps, with their full proofs
in the appendices.

\subsection{Organization and notation}
Section~\ref{sec:variational} defines the variational formula,
Section~\ref{sec:strategy} explains the proof, and
Section~\ref{sec:numerics} gives a certified trial bound and numerical
refinements. Complete proofs appear in the appendices: geometry in Appendix~\ref{sec:geometry},
probability cascades in Appendix~\ref{sec:cascades}, the pressure bounds
in Appendices~\ref{sec:lower} and~\ref{up:section}, the technical
upper-bound estimates in Appendices~\ref{ue:section} and~\ref{ud:section},
and the full variational and zero-temperature limits in
Appendix~\ref{sec:completion}. Appendix~\ref{app:trials} proves the
finite-family bounds used in Section~\ref{sec:numerics}.

We use the Euclidean inner product and norm, written $u\cdot v$ and
$\|u\|$, and natural logarithms. All spherical measures are probability
measures. A \emph{Gibbs measure} is the probability measure obtained by
normalizing an exponential weight such as $e^{-\beta H_W}$.
Gibbs brackets $\langle\cdot\rangle$ denote integration over independent
samples, called replicas, from the measure under consideration; $\E$
averages the remaining randomness. An \emph{overlap} is the inner product
of two unit directions, or the corresponding normalized inner product
when the directions are rescaled. Gaussian increments are independent
unless their covariance is specified. Constants denoted by $C$ may change
between occurrences. The notation $o_n(1)$ denotes a quantity tending to
zero as $n\to\infty$ with the other stated parameters fixed. Dimension
limits are taken at fixed temperature before the zero-temperature limit.

\section{The variational formula}\label{sec:variational}

Two unit input directions have overlap equal to their inner product.
Nearby directions share many row signs, so their active counts are
correlated. A finite trial models these correlations by a hierarchy of
overlap levels. The levels $q_i$ specify possible inner products, and
cumulative masses $x_i$ specify their probabilities.
Choose an integer $k\geq0$ and lists
$q=(q_0,\ldots,q_k)$ and $x=(x_1,\ldots,x_k)$ satisfying
\[
 0\leq q_0<\cdots<q_k<1,\qquad
 0<x_1<\cdots<x_k<1,
\]
where the second list is empty if $k=0$. Put $x_0=0$ and $x_{k+1}=1$.
The probability mass at level $q_i$ is $\pi_i=x_{i+1}-x_i$.
Define the residual variance $v$ and the positive quantities $\lambda_i$ by
\begin{equation}\label{eq:lambdas}
 v=1-q_k,\qquad
 \lambda_i=v+\sum_{j=i+1}^k x_j(q_j-q_{j-1}),\quad 0\leq i\leq k.
\end{equation}
The variance $v$ is the part of a fresh Gaussian row that remains
independent after the common hierarchical fields have been fixed.
The quantities $\lambda_i$ enter the entropy term below.
Starting from the one-row log weight, average successively over the
Gaussian increments. For a scalar field $h\in\R$ and a standard normal variable $G$, define
\begin{align}
 f_k(h)&=\log\E\exp\{-\beta\ind\{h+\sqrt vG>0\}\},\label{eq:terminal}\\
 f_{i-1}(h)&=\frac1{x_i}\log\E\exp\{x_i f_i(h+\sqrt{q_i-q_{i-1}}G)\},
       \quad i=k,\ldots,1.\label{eq:recursion}
\end{align}
The expectations in successive applications of the recursion use independent
normal variables. The row term $\mathcal J_\beta$, also called the pattern
term, measures the
contribution of one random row. The entropy term $\mathcal E$ is the
geometric cost of the trial profile. Set
\begin{align}
 \mathcal J_\beta(q,x)&=\E f_0(\sqrt{q_0}G),\label{eq:pattern}\\
 \mathcal E(q,x)&=\frac12\log v+\frac{q_0}{2\lambda_0}
  +\frac12\sum_{i=1}^k\frac1{x_i}
       \log\frac{\lambda_{i-1}}{\lambda_i},\label{eq:entropy}\\
 \mathcal P_{\alpha,\beta}(q,x)&=\mathcal E(q,x)+\alpha\mathcal J_\beta(q,x),
 \qquad P_*(\alpha,\beta)=\inf_{k,q,x}\mathcal P_{\alpha,\beta}(q,x).
 \label{eq:finite-functional}
\end{align}
The infimum is over all finite profiles satisfying the displayed conditions.
The terminal convolution is retained at every depth, including $k=0$.
In the cascade interpretation, it represents independent residual noise
even when two samples choose the same leaf. Appendix~\ref{sec:cascades}
gives that construction; the recursion above suffices to evaluate a trial.

For example, the profile $k=0$, $q_0=0$ has $\mathcal E=0$ and
\[
 \mathcal P_{\alpha,\beta}
 =\alpha\log\frac{1+e^{-\beta}}2.
\]
General profiles introduce correlated Gaussian fields before the final
one-row average. Theorem~\ref{main:pressure} asserts that optimization
over all their depths and parameters gives the exact limiting pressure.
No fixed number of levels is assumed to suffice.

The finite-profile expression agrees with
\cite[equations (47)--(49)]{Maillard2025Injectivity}. To describe its full
domain, associate the probability measure
$\zeta=\sum_{i=0}^k\pi_i\delta_{q_i}$ to a finite profile, where
$\delta_q$ is a unit point mass at $q$.
The row functional extends continuously to all probability measures on
$[0,1]$, while the entropy extends as a nonnegative lower semicontinuous
functional, possibly taking the value $+\infty$. Theorem~\ref{full:completion}
proves that, writing $\mathcal P([0,1])$ for the set of probability
measures on $[0,1]$,
\[
 P_*(\alpha,\beta)
 =\min_{\zeta\in\mathcal P([0,1])}
      \{\mathcal E(\zeta)+\alpha\mathcal J_\beta(\zeta)\}.
\]
Every minimizing measure has finite entropy and no atom at one, although
its support may approach one. Appendix~\ref{sec:completion} gives the
precise extension and its relation to the Parisi equation. This completion
is what permits the optimization over the entire hierarchy in the main
theorems.

\section{Proof of the threshold: main ideas}\label{sec:strategy}

The partition function averages over all directions, whereas injectivity
is decided by the worst one. We first identify the limit of the soft
minimum per input dimension,
then control its approach to the normalized minimum. Projection cavity
gives the lower pressure bound. Gaussian interpolation and response
calibration give the matching upper bound. A uniform estimate on the
volume of nearly minimizing directions completes the passage to the
injectivity threshold. The appendices contain the estimates and limiting
arguments.

\subsection{From injectivity to a minimum row count}
For a Gaussian matrix, every set of at most $n$ rows is linearly
independent almost surely. In a region where all row signs are fixed, the
ReLU map is linear, with the positive rows as its active rows. Fewer than
$n$ positive rows leave a direction along which the output does not
change. Conversely, if every direction has at least $n$ positive rows,
equality of two outputs forces equality of the inputs. Directions on row
hyperplanes require some care: their zero fields can all be perturbed to
negative fields without changing the other signs. Thus they introduce no
exception to the count criterion (Lemma~\ref{geo:criterion}).

The normalized minimum is also concentrated. Replacing a row changes
$\min H_W$ by at most one. For $m\geq1$ and $t>0$, Lemma~\ref{geo:rows} gives
\[
 \Prob\left\{\left|\frac{\min H_W}{n}
       -\frac{\E\min H_W}{n}\right|\geq t\right\}
\leq 2e^{-2n^2t^2/m}.
\]
For $m=0$, the minimum is identically zero.
Consequently, once the mean minimum has a limit different from one, its
position relative to one determines the injectivity phase. Establishing
and identifying that limit is the main problem.

\subsection{Projection cavity and the lower pressure bound}
The cavity method compares systems of nearby dimensions. Here the useful
comparison comes from restricting a measure on $S^{n+M-1}$ to a random
$n$-dimensional subspace, with $M$ fixed. Gaussian rows remain Gaussian
under this restriction. We add an isotropic Gaussian perturbation whose
contribution to the normalized pressure vanishes. The perturbation has
the same restriction property, so the comparison
preserves the joint law of the smaller system. It enforces identities on
the correlations between sampled directions at a vanishing pressure cost,
although it may change the Gibbs measure.

Haar averaging, meaning averaging over uniformly random orientations of
the subspace, recovers the full partition function from its restrictions.
The difference of their expected logarithms is a nonnegative reverse
relative entropy. Denote this restriction cost by $D_{n,M}$ and the
perturbed expected log partition in dimension $n$ by $a_n$, using a
Poisson number of rows with mean $\alpha n$. Let $\nu_n$ be the base
Gibbs measure on the larger sphere, and let $\Psi_{\alpha M}(\nu_n)$
be the expected log cost of adding independent fresh rows with Poisson
mean $\alpha M$. Appendix~\ref{sec:lower} gives the dimension-increment
identity
\[
 a_{n+M}-a_n=D_{n,M}+\Psi_{\alpha M}(\nu_n)+o_n(1).
\]
Both terms on the right use the same base measure. The error tends to
zero as $n\to\infty$ at fixed $M$.

A random subspace is the kernel of a Gaussian projection. Restricting to
that subspace therefore amounts to evaluating the projected Gibbs density
at zero, relative to the projected density of uniform spherical measure.
The resulting entropy and fresh-row terms form the projection cavity
functional, which can be expressed through overlaps. Its passage to a
limit requires control of the logarithm of that density. An overlap bound strictly below one leaves a positive
residual Gaussian variance; together with a weighted tail estimate, it
provides both the smoothing bounds and logarithmic integrability needed
for continuity (Lemmas~\ref{lo:cavitylimit} and~\ref{lo:continuity}).

To identify the limiting functional, sample directions independently
from the Gibbs measure and record their inner products. These samples
are usually called replicas, and their inner products are the overlaps.
The perturbation enforces the Ghirlanda--Guerra identities, which constrain
the joint law of these overlaps. Their representation and ultrametricity
consequences allow approximation by probability cascades, the hierarchical
measures whose Gaussian averages produce the recursion in
Section~\ref{sec:variational}.

Averaging the perturbation coefficients can produce a mixture of overlap
laws. We keep those laws as separate components, because the
Ghirlanda--Guerra identities are nonlinear in the law. The cavity
functional, in contrast, is affine under mixtures. Its value on a finite
cascade is exactly $M\mathcal P_{\alpha,\beta}(q,x)$. Continuity and finite
cascade approximation therefore bound each limiting dimension increment
below by $MP_*(\alpha,\beta)$. Telescoping the increments and removing the
perturbation proves
\[
 \liminf_n\Phi_{n,m_n}(\beta)\geq P_*(\alpha,\beta).
\]
This argument establishes the lower bound before a dimension limit is
known to exist.

\subsection{Response calibration and the upper pressure bound}
For the upper bound, fix a finite trial profile and interpolate between
the original row model and a hierarchical Gaussian model. The latter can
be evaluated by the recursion. The difficulty is to control the sign of
the interpolation error.

Three overlaps enter that error: the spin overlap $R$, the normalized
inner product of two input directions; the hierarchical overlap $Q$;
and a score overlap $C$. The score of a row is the derivative
of its smoothed log weight, so $C$ measures how similarly two replicas
respond to changes in their row fields. A compensated Gaussian
perturbation enforces joint overlap identities. Synchronization then
orders the spin and score coordinates together. Within each limiting
component law, the covariance satisfies
\[
 \operatorname{Cov}(R,C\mid Q=q_i)\geq0.
\]
When the component laws are mixed, their within-component covariances
are only part of the total covariance. A second term measures the
association between the component means, and can have either sign.
Controlling that term is the purpose of response calibration.

We introduce an auxiliary spin-field covariance profile and maximize the
coefficient-averaged pressure with its covariance correction. A logarithmic
barrier keeps the
covariance increments positive. These maximizing profiles are
deterministic, since the disorder and coefficient averages are taken
before optimization.

At a maximizer, the pressure is convex in the square roots of the
covariance increments and has a smooth upper supporting function.
Changing the perturbation coefficients changes its value by at most
$\varepsilon_n=O(s_n^2/n)$, where $s_n=n^\gamma$ and
$1/4<\gamma<1/2$. A small change in the value of a function does not generally
control its derivatives. Here convexity and the smooth support at a maximum provide
the missing control: secants over a distance $h$ give derivative error
$O(h+\varepsilon_n/h)$. Choosing $h$ of order
$\sqrt{\varepsilon_n}$ gives the required vanishing response error. Gaussian
differentiation identifies these derivatives with conditional spin
means. Thus those means agree across coefficient choices in the limit,
at each fixed positive barrier and with interpolation time bounded away
from one (Lemma~\ref{up:maximizers}).

The covariance of the component means consequently vanishes. The
within-component sign survives averaging and gives the required bound
for the optimized interpolation. The envelope argument integrates this
bound without choosing a differentiable maximizing path. The spherical
endpoint estimate is proved in Appendix~\ref{ue:section};
Appendix~\ref{ud:section} treats differentiation of the row-dependent
perturbation. Together they yield, for every finite trial,
\[
 \limsup_n\Phi_{n,m_n}(\beta)
       \leq\mathcal P_{\alpha,\beta}(q,x).
\]
Taking the infimum and combining with the cavity lower bound proves
Theorem~\ref{main:pressure}.

\subsection{Uniform volume and the energy-one crossing}
At a fixed dimension, increasing $\beta$ makes the soft minimum favor
directions minimizing $H_W$. To use the pressure limit, however,
we need a comparison that is uniform in the dimension. Positivity of the
volume of a minimizing chamber does not suffice: that volume could be
superexponentially small in $n$.

Proposition~\ref{geo:volume} gives the needed quantitative estimate.
For each $\delta>0$ and densities in a fixed compact subinterval of
$(1,\infty)$, there is a constant $C_\delta<\infty$, independent of $n$,
such that, with probability tending to one,
\[
 \mu_n\{u:H_W(u)\leq\min H_W+\delta n\}\geq e^{-C_\delta n}.
\]
The proof uses singular-value bounds that hold simultaneously for all
relevant row subsets. They allow a small displacement of any chosen
direction so that all but $\delta n$ of its initially nonpositive rows
acquire a negative margin. A spherical cap then stays within the desired
sublevel set. Since the matrix estimates are uniform, the chosen direction
may be a minimizer selected after observing $W$.

Integrating the Gibbs weight over this cap gives
\[
 0\leq-\frac{\Phi_{n,m}(\beta)}\beta
        -\frac{\E\min H_W}{n}
 \leq\delta+\frac{C_\delta}\beta+o_n(1).
\]
We first take $n\to\infty$, then $\beta\to\infty$, and finally
$\delta\downarrow0$. The pressure formula now identifies the limiting
minimum with $e(\alpha)$. Elementary separability and row comparison
cover the remaining small densities.

It remains to show that the level one is crossed at a unique density.
The variational formula makes $e$ convex, nondecreasing and
$1/2$-Lipschitz, with $e(0)=0$. Elementary bounds give $e(4)\leq1$ and
$e(16)\geq1$, so a crossing exists. Convexity gives the strict inequalities
\[
 e(\alpha)\leq\frac{\alpha}{\alpha_\star}<1
       \quad(\alpha<\alpha_\star),\qquad
 e(\alpha)\geq\frac{\alpha}{\alpha_\star}>1
       \quad(\alpha>\alpha_\star).
\]
Concentration and the geometric criterion then give both injectivity
phases. Appendix~\ref{sec:completion} carries out these limits for every
deterministic aspect-ratio sequence and proves convergence also in $L^1$.

\section{Finite-hierarchy trials}\label{sec:numerics}

The Gaussian recursion provides finite-dimensional trials for the
variational problem. Their limiting values give lower bounds on the ground
energy and hence
upper bounds on the injectivity threshold. We give a rational certificate
for one trial and then evaluate progressively deeper hierarchies.

For a fixed $\chi>0$ and fixed rescaled masses $m_i>0$, the families
scale the last residual variance as $\chi/\beta$ and the cumulative
masses as $m_i/\beta$. In the limit $\beta\to\infty$, the
one-row cost becomes the clipped quadratic $\psi_\chi$ below.
The parameters in the following formulas describe this limiting family;
at finite $\beta$, the terminal overlap is $1-\chi/\beta<1$.
Choose an integer $k\geq0$, overlap levels
$0\leq q_0<\cdots<q_k=1$ and masses $0<m_1<\cdots<m_k$.
With $h_+=\max\{h,0\}$ denoting the positive part of $h$, put
\[
 \ell_i=\chi+\sum_{j=i+1}^k m_j(q_j-q_{j-1}),\qquad
 \psi_\chi(h)=\min\{h_+^2/(2\chi),1\}.
\]
For $k=0$ these conventions mean $q_0=1$ and $\ell_0=\chi$.
With $G$ standard normal and independent at each step, define the bounded functions
\begin{align*}
 g_k(h)&=-\psi_\chi(h),\\
 g_{i-1}(h)&=\frac1{m_i}\log\E\exp\{
       m_i g_i(h+\sqrt{q_i-q_{i-1}}G)\},\qquad i=k,\ldots,1,
\end{align*}
and set
\begin{align}
 E_0&=\frac{q_0}{2\ell_0}
    +\frac12\sum_{i=1}^k\frac1{m_i}\log\frac{\ell_{i-1}}{\ell_i},
    \label{num:entropy}\\
 J_0&=\E g_0(\sqrt{q_0}G),\qquad
 B_\alpha=-E_0-\alpha J_0.\label{num:trial}
\end{align}
Here $E_0$ and $J_0$ are the limiting entropy and row terms, respectively,
and $B_\alpha$ is the resulting energy bound. Write
$q=(q_0,\ldots,q_k)$ and $\boldsymbol m=(m_1,\ldots,m_k)$.

\begin{proposition}\label{num:family}
For every such finite family and every finite $\alpha\geq0$,
$B_\alpha\leq e(\alpha)$. In particular, its crossing
\begin{equation}\label{num:crossing}
 C(q,\boldsymbol m,\chi)=\frac{1+E_0}{-J_0}
\end{equation}
satisfies $\alpha_\star\leq C(q,\boldsymbol m,\chi)$.
\end{proposition}
The proof is given in Appendix~\ref{app:trials}.

\subsection{A certified one-step trial}
A single fixed trial suffices to prove the following bound.
\begin{corollary}\label{num:seven}
The critical density satisfies $\alpha_\star<7$.
\end{corollary}
The trial uses $q_0=39/50$, $\chi=1/20$ and $m_1=16/5$.
A rational enclosure gives $B_7>1.06$, so Proposition~\ref{num:family}
implies $e(7)>1$.
Appendix~\ref{app:trials} derives the one-dimensional integral and bounds
its arithmetic, integration and Gaussian-tail errors. The ancillary
checker reproduces all the rational bounds; its files and execution
command are given below.

\subsection{Numerical refinement}
We evaluated the recursions in double precision and searched for small
values of \eqref{num:crossing}. For $k\geq1$, the first Gaussian convolution of
$e^{-m_k\psi_\chi}$ is evaluated analytically, as in Appendix~\ref{app:trials}.
Later convolutions use Gauss--Hermite quadrature and monotone cubic
interpolation on a one-dimensional field grid. The outer Gaussian integral
uses separate panels at the field knots and at Gaussian-scale knots.
Profiles are parametrized by positive overlap and mass increments.
Bounded L-BFGS-B searches use two initial profiles and, when available,
a refinement of the preceding trial, with at most 150 iterations per start.
The eight- and sixteen-level searches reached this iteration limit.

\begin{table}[htbp]
\centering
\caption{Feasible-family trial crossings, reevaluated with 3201 field points and quadrature order 256. The search uses 401 points and order 24. $\Delta C$ is the signed change on reevaluation. The listed decimal values are numerical observations.}
\label{num:table}
\begin{tabular}{r r r r r}
\hline
$k$ & $C$ & $q_0$ & $\chi$ & $\Delta C$ \\
\hline
0 & 7.647695 & 1.00000 & 0.63036 & $0$ \\
1 & 6.715697 & 0.77723 & 0.05330 & $-5.15\cdot10^{-8}$ \\
2 & 6.700392 & 0.74113 & 0.03566 & $+1.50\cdot10^{-6}$ \\
4 & 6.698686 & 0.71979 & 0.03079 & $-1.29\cdot10^{-7}$ \\
8 & 6.698145 & 0.71907 & 0.02785 & $-3.70\cdot10^{-7}$ \\
16 & 6.698073 & 0.72076 & 0.02530 & $+3.28\cdot10^{-6}$ \\
\hline
\end{tabular}
\end{table}

The listed crossings are values of feasible trial families. Global
optimization and the displayed decimal places are not certified.
Grid resolution, Gaussian quadrature order and field truncation were varied
separately while holding each trial fixed. The table reports the change
between the search evaluation and the finest reevaluation; these changes
measure numerical sensitivity and do not bound the optimization error.
The replica-symmetric and one-step values reproduce the corresponding
calculations in \cite{Maillard2025Injectivity}. Higher-level trials are
consistent with the numerical prediction near $6.698$, without proving a
two-sided decimal enclosure for $\alpha_\star$.

The temperature comparison in Figure~\ref{num:figure} concerns a fixed
family. At any fixed $\beta$, a feasible profile gives
$-\mathcal P_{\alpha,\beta}/\beta\leq-P_*(\alpha,\beta)/\beta$, while
$e(\alpha)\leq-P_*(\alpha,\beta)/\beta$. These two inequalities impose
no ordering between a finite-temperature trial and $e(\alpha)$.
The one-sided conclusion in Proposition~\ref{num:family} uses the
justified limit of an entire admissible family. Corollary~\ref{num:seven}
is obtained from that limit and the rational certificate.

\paragraph{Code and data availability.}
The paper's source package includes the ancillary directory \path{anc/}.
From that directory, run the exact-arithmetic checker using Python's
standard library:
\begin{quote}
\small
\begin{verbatim}
python verify_one_step.py --config one_step_parameters.json \
    --output certificate.json
\end{verbatim}
\end{quote}
Saved profiles, resolution studies and finite-temperature values are in
\path{trial_data.json}; \path{plot_profile_experiments.py} generates
Table~\ref{num:table} and Figure~\ref{num:figure} from these data. The table
rounds $C$ to six decimal places; the data retain full precision.
See \path{README.txt} for dependencies and reproduction commands.

\begin{figure}[htbp]
 \centering
 \includegraphics[width=\textwidth]{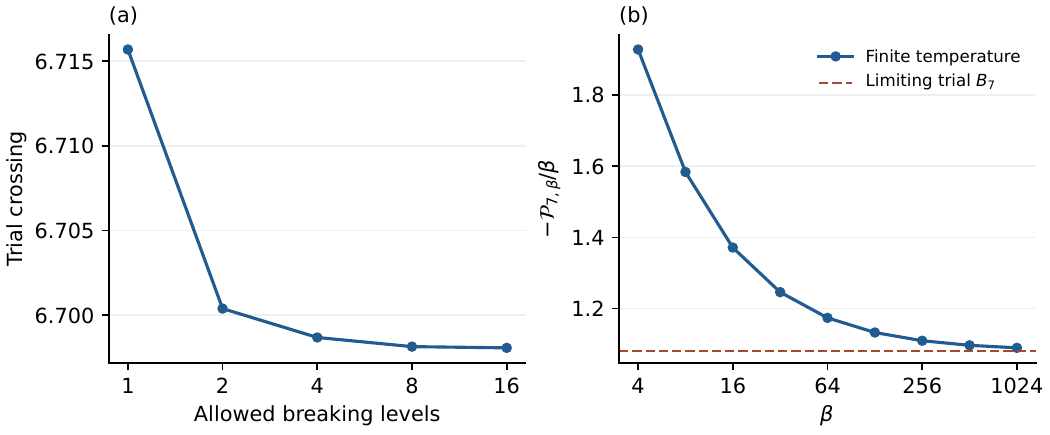}
 \caption{Finite-family trial evaluation. (a) The trial crossings decrease
 from $6.7157$ at one breaking level to $6.6981$ at sixteen levels. (b) Finite-temperature values for the rational one-step family
 of Corollary~\ref{num:seven}, decreasing toward its limiting trial value
 $B_7\approx1.081$, shown by the dashed horizontal line. Floating-point refinement is separate
 from the rational certificate \eqref{num:certificate}.}
 \label{num:figure}
\end{figure}

\clearpage
\appendix
\section{Geometry and the zero-temperature approximation}\label{sec:geometry}

Write
\[
 E_{n,m}(W)=\frac1n\min_{u\in S^{n-1}}H_W(u),\qquad
 h(m,n)=\E\min_{u\in S^{n-1}}H_W(u).
\]
The minimum is attained, since $H_W$ has a finite, nonempty set of values.
It is measurable: for each integer $r$, the event $\min H_W\leq r$ is a
finite union of events asserting that a specified collection of homogeneous
weak inequalities has a solution on the compact unit sphere. Each such
event is closed in the matrix entries.

\begin{lemma}\label{geo:criterion}
Almost surely,
\[
 f_W\text{ is injective on }\R^n
 \quad\Longleftrightarrow\quad E_{n,m}(W)\geq1.
\]
\end{lemma}
\begin{proof}
Gaussian rows are in general position: every existing set of at most $n$
rows is linearly independent. For an index set $I$, let $W_I$ be
the matrix consisting of those rows, and let $\mathbf1$ denote the
all-ones vector of the required dimension. At a nonzero $u$, at most
$n-1$ rows vanish.
For their index set $I$, solve $W_Iv=-\mathbf1$ and perturb $u$ to
$u+tv$, with $t>0$ small. The zero fields become negative and all other
signs are unchanged. Thus each boundary count is attained in an open
chamber. If that count is below $n$, the active rows have a nonzero common
null vector. A sufficiently small displacement along it leaves the ReLU
output unchanged.

Conversely, suppose every spherical count is at least $n$ and
$f_W(a)=f_W(b)$. At $z=(a+b)/2$, every row with nonnegative field
annihilates $a-b$: equal positive outputs give equal fields, while equal
zero outputs and a nonnegative midpoint force both fields to vanish.
These rows span $\R^n$, by general position if $z\ne0$, and by full
column rank if $z=0$. Hence $a=b$. If $m<n$, rank deficiency gives
noninjectivity and the count condition also fails.
\end{proof}

\begin{lemma}[Row comparison]\label{geo:rows}
For integers $s\geq0$ and $1\leq m\leq M$,
\begin{equation}\label{eq:row-comparison}
 0\leq h(m+s,n)-h(m,n)\leq s/2,
 \qquad \frac{h(m,n)}m\leq\frac{h(M,n)}M.
\end{equation}
For $m\geq1$ and $t>0$,
\begin{equation}\label{eq:count-concentration}
 \Prob\left\{\abs{E_{n,m}-h(m,n)/n}\geq t\right\}
 \leq2\exp\left(-\frac{2n^2t^2}{m}\right).
\end{equation}
\end{lemma}
\begin{proof}
Adding rows can only increase the minimum. Choose a minimizing direction
measurably from a fixed countable dense set, which is possible by the
chamber argument above. Each independent new Gaussian row has a positive
field there with probability $1/2$. This proves the first inequality.

For a fixed $M$-row matrix and each $m$-element subset $I$ of rows, let
$H_I$ denote its count. Then
\[
 \sum_{\abs I=m}\min_u H_I(u)
 \leq\min_u\sum_{\abs I=m}H_I(u)
 =\binom{M-1}{m-1}\min_u H_W(u).
\]
Taking expectations proves the second inequality. Finally, replacing one
row changes $\min H_W$ by at most one. The bounded-differences inequality
gives \eqref{eq:count-concentration}.
\end{proof}

Only rough bounds are needed to locate an energy-one crossing.
\begin{lemma}\label{geo:anchors}
Along every sequence $m_n/n\to4$,
$\limsup_n h(m_n,n)/n\leq1$. Along every sequence $m_n/n\to16$,
$p_{m_n,n}\to1$.
\end{lemma}
\begin{proof}
The central hyperplane count \cite{Cover1965Geometrical,Wendel1962Problem}
is
\begin{equation}\label{eq:chamber-count}
 C(m,n)=2\sum_{j=0}^{n-1}\binom{m-1}{j}.
\end{equation}
It follows by inserting the last hyperplane: the additional chambers are
the chambers of the induced general-position arrangement in dimension
$n-1$, giving $C(m,n)=C(m-1,n)+C(m-1,n-1)$, with $C(1,n)=C(m,1)=2$.
Independent row-sign symmetry shows that each fixed sign pattern is
feasible with probability $2^{-m}C(m,n)$.

Fix $\varepsilon\in(0,1)$ and expose the first
$k=\lfloor(2-\varepsilon)n\rfloor$ rows. The probability that they admit
a strictly negative direction is
\[
 2^{-(k-1)}\sum_{j=0}^{n-1}\binom{k-1}{j}=1-o(1).
\]
Choose such a direction measurably from rational directions, using a fixed
default direction if no separator exists. The remaining fields are
conditionally independent standard normals. Consequently
\[
 h(m_n,n)\leq\frac{m_n-k}{2}+k\,o(1).
\]
First let $n\to\infty$, then $\varepsilon\downarrow0$.

For the upper anchor, let $S=\sum_{j=0}^{n-1}\binom mj$. The expected
number of feasible sign patterns with fewer than $n$ positive coordinates
is at most $2^{1-m}S^2$, by \eqref{eq:chamber-count}. For $m\geq2n$,
the binomial entropy bound gives
\[
 S\leq\exp\{m\,\mathsf h((n-1)/m)\},\qquad
 \mathsf h(p)=-p\log p-(1-p)\log(1-p).
\]
Moreover $\mathsf h(p)\leq p\log(e/p)$, so
\[
 2\mathsf h(1/16)\leq\frac{\log(16e)}8<\log2.
\]
The last inequality follows from $\log2>1/2$. Thus the expected number
decays exponentially when $m/n\to16$. Markov's inequality and
Lemma~\ref{geo:criterion} prove the claim.
\end{proof}

\subsection{A uniform lower bound on sublevel volumes}
The discontinuity of the cost requires a volume estimate uniform over
directions selected after observing $W$.

\begin{proposition}\label{geo:volume}
Fix $A\geq2$ and $\eta\in(0,1/8]$. There are constants
$\rho=\rho(A,\eta)>0$ and $C=C(A,\eta)<\infty$ such that, whenever
\[
 n\geq\eta^{-1},\qquad
 \lceil(1+\eta)n\rceil\leq m\leq An,
\]
with probability at least $1-4e^{-3n}$ the following holds simultaneously
for every $u\in S^{n-1}$: there exists $q\in S^{n-1}$ such that
\begin{equation}\label{eq:cap-count}
 H_W(z)\leq H_W(u)+4\eta n
 \quad\text{whenever }\norm{z-q}\leq\rho,
 \qquad
 \mu_n\{z:\norm{z-q}\leq\rho\}\geq e^{-Cn}.
\end{equation}
\end{proposition}
\begin{proof}
Let $r=\lceil(1+\eta)n\rceil$, $k=\lfloor(1-\eta)n\rfloor$ and set
\begin{align*}
 K&=\sqrt{8((A+1)\log9+3)},\qquad
 D=A\log2+\log(3K)+A\log(2\sqrt e),\\
 s&=e^{-(D+3)/\eta},\qquad T=\frac{s}{2\sqrt A},\qquad
 \kappa=\frac{sT\sqrt\eta}{4K},\qquad
 \rho=\frac{\kappa\sqrt\eta}{4K}.
\end{align*}
Write $\|A\|_{\mathrm{op}}=\sup_{\|z\|=1}\|Az\|$ for the operator
norm and $\sigma_{\min}(A)=\inf_{\|z\|=1}\|Az\|$ for the least
singular value of a matrix with at least as many rows as columns.
We first prove that, except on an event of probability $4e^{-3n}$,
\begin{equation}\label{eq:matrix-event}
 \norm W_{\mathrm{op}}\leq K\sqrt n,\qquad
 \sigma_{\min}(W_I^T)\geq s\sqrt n\ (\abs I=k),\qquad
 \sigma_{\min}(W_J)\geq s\sqrt n\ (\abs J=r).
\end{equation}
An $a$-net is a finite set within distance $a$ of every point of the
sphere. A Euclidean unit sphere in dimension $d$ has such a net of cardinality
at most $(1+2/a)^d$. For a standard Gaussian vector $G\in\R^d$,
\begin{equation}\label{eq:gaussian-small-ball}
 \Prob\{\norm G\leq R\}\leq(\sqrt e R/\sqrt d)^d;
\end{equation}
indeed, apply Markov's inequality to $e^{-\lambda\norm G^2}$ with
$\lambda=d/(2R^2)$. Applying $1/4$-nets to the bilinear expression for
$\norm W_{\mathrm{op}}$ yields
\[
 \Prob\{\norm W_{\mathrm{op}}>K\sqrt n\}
 \leq2\exp\{(m+n)\log9-K^2n/8\}\leq2e^{-3n}.
\]
On the complementary event, failure of the bound for a fixed $I$ in
\eqref{eq:matrix-event} is detected on an $s/K$-net of $S^{k-1}$ with
image norm below $2s\sqrt n$. A union bound over subsets and
\eqref{eq:gaussian-small-ball} give
\[
 \Prob\{\norm W_{\mathrm{op}}\leq K\sqrt n,
             \exists I:\sigma_{\min}(W_I^T)<s\sqrt n\}
 \leq2^m(3K)^k(2\sqrt e)^n s^{n-k}\leq e^{nD}s^{\eta n}=e^{-3n}.
\]
For $J$, use an $s/K$-net of $S^{n-1}$. The corresponding bound is
$2^m(3K)^n(2\sqrt e)^r s^{r-n}\leq e^{-3n}$, proving
\eqref{eq:matrix-event}. The estimate for $I$ also holds for every smaller
nonempty subset, by extending it to size $k$ and padding coefficient
vectors with zeros.

Work on this event and fix $u\in S^{n-1}$. There are at most $r-1$ rows
with $\abs{w_i\cdot u}\leq T$: otherwise a subset $J$ of size $r$ would
satisfy $\norm{W_Ju}\leq T\sqrt{An}=s\sqrt n/2$. Choose $I$ of size
$\min\{k,\abs N\}$ from
$N=\{i:-T\leq w_i\cdot u\leq0\}$. The discarded set has size
\[
 \abs{N\setminus I}\leq\max\{0,r-1-k\}<2\eta n+1\leq3\eta n.
\]
Let $v$ be the minimum-norm solution of $W_Iv=-\kappa\mathbf1$, or zero
if $I$ is empty. Then
\[
 \norm v\leq\kappa/s<1/4,
 \qquad w_i\cdot(u+v)\leq-\kappa\quad(i\in I).
\]
Among rows with $w_i\cdot u<-T$, at most
\[
 \frac{4\norm{Wv}^2}{T^2}
 \leq\frac{4K^2n\kappa^2}{s^2T^2}=\frac{\eta n}4
\]
have $w_i\cdot v>T/2$. All other such rows have field at most
$-T/2\leq-\kappa$ at $u+v$. Thus, except for at most
$3\eta n+\eta n/4$ rows, every initially nonpositive field becomes at
most $-\kappa$.

Normalize $q=(u+v)/\norm{u+v}$. Since $\norm{u+v}<2$, the protected
fields at $q$ are at most $-\kappa/2$. For any $z$ with
$\norm{z-q}\leq\rho$, at most
\[
 \frac{4\norm{W(z-q)}^2}{\kappa^2}
 \leq\frac{4K^2n\rho^2}{\kappa^2}=\frac{\eta n}4
\]
protected fields become positive. Counting the three exceptional sets proves
the first assertion of \eqref{eq:cap-count}.

The cap has angular radius $\theta=2\arcsin(\rho/2)$. Since $\rho<1$,
polar integration gives
\[
 \mu_n\{z:\norm{z-q}\leq\rho\}
 =\frac{\int_0^\theta\sin^{n-2}t\,\dd t}
        {\int_0^\pi\sin^{n-2}t\,\dd t}
 \geq\frac\theta{2\pi}\sin^{n-2}(\theta/2)
 \geq\left(\frac\rho{2\pi}\right)^n.
\]
Take $C=\log(2\pi/\rho)$.
\end{proof}

\begin{corollary}\label{geo:bridge}
Fix $1<a<A$ with $A\geq2$. For every $\delta>0$ there is
$C_{a,A,\delta}<\infty$ such that for all sufficiently large $n$,
all $an\leq m\leq An$, and every $\beta>0$,
\begin{equation}\label{eq:bridge}
 0\leq-\frac{\Phi_{n,m}(\beta)}\beta-\frac{h(m,n)}n
 \leq\delta+\frac{C_{a,A,\delta}}\beta+2Ae^{-3n}.
\end{equation}
\end{corollary}
\begin{proof}
Let $T_{n,m}(\beta)=-\log Z_{n,m}(\beta)/(\beta n)$. Directly from the
partition function, $T_{n,m}(\beta)\geq E_{n,m}$. Jensen's inequality
and the fact that each nonzero row is positive on half the sphere give
$T_{n,m}(\beta)\leq m/(2n)$. Apply Proposition~\ref{geo:volume} with
$\eta=\min\{\delta/4,(a-1)/2,1/8\}$ at a minimizing direction.
For $n\geq\eta^{-1}$ we have
$\lceil(1+\eta)n\rceil\leq(1+2\eta)n\leq an$, so its row-count
hypotheses hold.
On its event, integration over the cap gives
$T_{n,m}(\beta)\leq E_{n,m}+\delta+C_{a,A,\delta}/\beta$. On the exceptional
event the difference is at most $A/2$. Taking expectations proves
\eqref{eq:bridge}.
\end{proof}

\section{Probability cascades and overlap arrays}
\label{sec:cascades}

We record the cascade construction and representation results used in both
pressure bounds. Throughout this section, $(q,x)$ is a finite profile as in
\eqref{eq:lambdas}--\eqref{eq:finite-functional}, of depth $k$ and with
masses $\pi_i=x_{i+1}-x_i$.

\begin{definition}[Finite probability cascade]\label{rpc:construction}
For $i=1,\ldots,k$, at every vertex of depth $i-1$, independently place
a Poisson process on $(0,\infty)$ with intensity
$x_i t^{-1-x_i}\,\dd t$, enumerating its points in decreasing order.
For a leaf
$a=(a_1,\ldots,a_k)\in\mathbb N^k$, let $\xi_a$ be the product of the
points on its ancestral path, and put
\[
 w_a=\frac{\xi_a}{\sum_b\xi_b},\qquad
 Q_{ab}=q_{|a\wedge b|}.
\]
Here $|a\wedge b|$ is the depth of the last common ancestor, so that
$Q_{aa}=q_k$.  For $k=0$ there is one leaf of weight one and $Q=q_0$.
\end{definition}

The total weight is finite and positive almost surely.  Indeed, a Poisson
sum of index $x_i$ is a positive stable variable, and its positive moments
of orders below $x_i$ are finite; the inequalities $x_{i-1}<x_i$ allow
this observation to be iterated from the leaves to the root.  If
$a^1,a^2,\ldots$ are sampled independently from $(w_a)$, the associated
unit-diagonal overlap array is
\begin{equation}\label{rpc:unitarray}
 R_{\ell\ell'}=Q_{a^\ell a^{\ell'}}\quad(\ell\ne\ell'),
 \qquad R_{\ell\ell}=1,
 \qquad \Prob(R_{12}=q_i)=\pi_i.
\end{equation}
These are the standard finite Ruelle probability cascades; see
\cite[Sections 2.2--2.4 and Theorem 2]{Arguin2007RPC}.

A Gaussian field of covariance $Q$ is obtained by assigning an independent
Gaussian of variance $q_0$ to the root and independent increments of
variance $q_i-q_{i-1}$ to vertices of depth $i$.  Denote its value at leaf
$a$ by $z_a$.  A Gaussian sequence with covariance
\eqref{rpc:unitarray} is then
\begin{equation}\label{rpc:residual}
 X^\ell=z_{a^\ell}+\sqrt{1-q_k}\,Y^\ell,
\end{equation}
where the $Y^\ell$ are independent standard Gaussians.  The residual
Gaussians are independent even when two replicas select the same leaf.

\begin{lemma}[Normalized marking recursion]\label{rpc:marking}
Decorate the cascade by a root variable and by random variables at its
vertices, independently of its Poisson points.  Conditional on the
decorations along a path, the child decorations and their descendant
subtrees are independent and identically distributed.  A leaf mark
$F_a>0$ may be any measurable function of the decorations along that
leaf's path.  On one such path set $Y_k=\log F_a$ and, successively,
\begin{equation}\label{rpc:recursion}
 Y_{i-1}=\frac1{x_i}\log\E_i\exp(x_iY_i),
 \qquad i=k,\ldots,1,
\end{equation}
where $\E_i$ integrates the depth-$i$ decoration conditionally on its
ancestors.  If these conditional moments are finite and positive and
the resulting logarithms are integrable, then
\begin{equation}\label{rpc:markingidentity}
 \E\log\sum_a w_a F_a=\E Y_0.
\end{equation}
The root is averaged by ordinary expectation.  Decorations may be
random functions evaluated at later descendant decorations, provided
the stated conditional independence holds.
\end{lemma}

\begin{proof}
For Poisson points $(t_j)$ of intensity $x t^{-1-x}\,\dd t$ and iid
positive marks $B_j$, independent of those points, the Laplace functional
gives
\[
 \E\exp\Bigl(-s\sum_jt_jB_j\Bigr)
 =\exp\{-\Gamma(1-x)s^x\E B^x\}.
\]
Here $\Gamma$ denotes Euler's gamma function.
Thus the marked sum has the law of $(\E B^x)^{1/x}S_x$, where
$S_x=\sum_jt_j$.  Consequently the difference between its expected
logarithm and $\E\log S_x$ is $x^{-1}\log\E B^x$.

Apply this calculation conditionally, starting at depth $k$.  At the
next level, the marked subtree sum is a stable variable multiplied by
the scale furnished by \eqref{rpc:recursion}.  Its $x_{k-1}$ moment
exists because $x_{k-1}<x_k$.  The factor $\E S_{x_k}^{x_{k-1}}$
also occurs in the unmarked subtree and therefore cancels in the
ratio of scales.  Continuing in this way proves that the logarithmic
scale ratio of the full marked and unmarked sums is $Y_0$.
Taking the difference of their expected logarithms proves
\eqref{rpc:markingidentity}.  This calculation uses marginal laws
of the two sums and does not require them to be independent.
The assumed logarithmic integrability justifies each subtraction.
Bounded positive marks with bounded logarithms satisfy these
conditions.  For unbounded logarithms the conditional moments and
logarithmic integrability are checked in each application below.
\end{proof}

An \emph{overlap law} will mean the law of a weakly exchangeable,
positive semidefinite array $R=(R_{\ell\ell'})_{\ell,\ell'\geq1}$
with $R_{\ell\ell}=1$. Weak exchangeability means invariance under
simultaneous finite permutations of the replica labels, and positive
semidefiniteness is required for every finite submatrix.
We use weak convergence on the compact space
of such array laws.  For a law $\mathscr L$, write $\mathscr L(f)$ for
the integral of $f$.  Its Ghirlanda--Guerra identities are
\begin{equation}\label{rpc:GG}
 \mathscr L\bigl(f\psi(R_{1,r+1})\bigr)
 =\frac1r\mathscr L(f)\mathscr L\bigl(\psi(R_{12})\bigr)
  +\frac1r\sum_{\ell=2}^r
       \mathscr L\bigl(f\psi(R_{1\ell})\bigr),
 \qquad r\geq2,
\end{equation}
for every bounded measurable function $f$ of the first $r$ replicas'
off-diagonal overlaps and every bounded measurable $\psi$.

\begin{proposition}[Representation and finite approximation]
\label{rpc:representation}
Every overlap law has a representation
\begin{equation}\label{rpc:DS}
 R_{\ell\ell'}=h^\ell\cdot h^{\ell'}
       +\ind\{\ell=\ell'\}(1-\norm{h^\ell}^2),
\end{equation}
where, conditionally on a random probability measure $\nu$ on the unit
ball of a separable Hilbert space, the $h^\ell$ are iid with law $\nu$.
If the law satisfies \eqref{rpc:GG}, its directing measure is almost
surely supported on the sphere of squared radius
$q_* =\sup\operatorname{supp}\operatorname{Law}(R_{12})$, and its
off-diagonal overlaps are ultrametric, meaning
$R_{23}\geq\min\{R_{12},R_{13}\}$ almost surely. The entire array law is
determined by the distribution of $R_{12}$.

In particular, if $0\leq R_{12}\leq r_*<1$ almost surely and
\eqref{rpc:GG} holds, the law is the weak limit of finite cascade
laws \eqref{rpc:unitarray} with $q_k\leq r_*$.
\end{proposition}

\begin{proof}
The representation is the Dovbysh--Sudakov theorem
\cite[Theorem 1, p.~244]{Panchenko2014SK}.  The sphere-support
conclusion follows by applying \cite[Theorem 6, p.~263]{Panchenko2014SK}
to the Gram array $(h^\ell\cdot h^{\ell'})$: its off-diagonal entries
agree with $R$, and the theorem's hypothesis (86) tests only those
entries.  The resulting Gram diagonal is the constant $q_*$.
The original array retains diagonal one through the separate correction
$1-q_*$ in \eqref{rpc:DS}; neither diagonal enters the identities.
Ultrametricity then follows from
\cite[Theorem 1]{Panchenko2013Ultrametricity}, and uniqueness from
\cite[Theorem 11, pp.~275--277]{Panchenko2014SK}.  The unit-diagonal
array formulation and its cascade identification are also stated in
\cite[Theorem 1.1 and the discussion following it]{Panchenko2022Ultrametricity}.

For completeness, the finite approximation preserves all the relevant
hypotheses.  Choose nondecreasing finite-range functions
$\kappa_j:[0,r_*]\to[0,r_*]$ converging uniformly to the identity.
Replace off-diagonal overlaps by $\kappa_j(R_{\ell\ell'})$ and retain
diagonal one.  Ultrametricity makes each threshold relation
$R_{\ell\ell'}\geq t$ an equivalence relation.  Its matrix, with
diagonal one, consists of all-one blocks and is positive semidefinite.
The transformed array is a nonnegative linear combination of these
matrices, an all-one matrix, and a nonnegative diagonal correction.
It is therefore positive semidefinite.  Exchangeability persists,
and \eqref{rpc:GG} persists by composition of its test functions.
The finite-support case of the cited uniqueness theorem identifies
this law with the finite cascade having its two-replica marginal.
Deleting levels of zero mass gives the strict inequalities in
Definition~\ref{rpc:construction}; a point mass gives $k=0$.
Uniform approximation of each off-diagonal entry proves weak
convergence of the array laws.
\end{proof}

\section{The pressure lower bound}
\label{sec:lower}

\begin{theorem}\label{lo:pressure}
For every finite $\alpha\geq0$, every $\beta>0$, and every deterministic
sequence of row counts $m_n/n\to\alpha$,
\[
 \liminf_{n\to\infty}\Phi_{n,m_n}(\beta)
 \geq P_*(\alpha,\beta).
\]
\end{theorem}

We prove the theorem with a Poisson number of rows and then transfer it
to deterministic counts.  Write
\[
 A_n=\E\log Z_{n,P_n}(\beta),\qquad
 P_n\sim\operatorname{Pois}(\alpha n).
\]
Throughout the proof $\alpha$ and $\beta$ are fixed.  Every limit involving
a cavity dimension $M$ is taken with $M$ fixed.

\subsection{An isotropic perturbation}

Let $(\omega_p)_{p\geq1}$ be independent uniform random variables on
$[0,1]$.  Independently of the rows, take independent centered Gaussian
processes $g_{p,n}$ on $S^{n-1}$ with covariance
\[
 \E g_{p,n}(u)g_{p,n}(u')=(u\cdot u')^p.
\]
They can be realized by contracting iid Gaussian $p$-tensors with
$u^{\otimes p}$.  Set
\begin{equation}\label{lo:perturbation}
 G_{n,\omega}(u)=\sum_{p\geq1}2^{-p}\omega_p g_{p,n}(u),
 \qquad C_\omega(q)=\sum_{p\geq1}4^{-p}\omega_p^2q^p,
 \qquad s_n=n^\gamma,\quad \tfrac14<\gamma<\tfrac12.
\end{equation}
This process has a continuous version.  Indeed, Gaussian comparison
bounds the expected supremum norm of $g_{p,n}$ by $C\sqrt{pn}$, and
$\sum_p2^{-p}\sqrt p<\infty$ gives uniform convergence of the series
almost surely and in expected supremum norm.  Conditional on $\omega$,
its law is orthogonally invariant; restriction to an $N$-dimensional
subspace has exactly the law of $G_{N,\omega}$.

Let $a_n(t,\omega)$ be the expected log partition with
$\operatorname{Pois}(\alpha n)$ rows and log density
$-\beta H+\sqrt t\,G_{n,\omega}$ relative to uniform spherical
measure $\mu_n$, where $H$ counts the positive row inner products.
For the resulting Gibbs measure $\nu$, brackets denote averages over
independent samples $u^1,u^2,\ldots$ from $\nu$, called replicas;
their overlaps are $R_{\ell\ell'}=u^\ell\cdot u^{\ell'}$.
Define $a_n=\E_\omega a_n(s_n^2,\omega)$.  Gaussian interpolation gives
\begin{equation}\label{lo:amplitude}
 \frac{\partial a_n(t,\omega)}{\partial t}
 =\frac12\bigl(C_\omega(1)-\E\langle C_\omega(R_{12})\rangle\bigr),
 \qquad
 0\leq\frac{\partial a_n(t,\omega)}{\partial t}\leq\frac16.
\end{equation}
For the first inequality, $C_\omega(q)\leq C_\omega(1)$ on $[-1,1]$.
For the second, each Gibbs measure satisfies
\[
 \langle R_{12}^p\rangle
 =\left\|\int u^{\otimes p}\,\dd\nu(u)\right\|^2\geq0,
 \qquad C_\omega(1)\leq\tfrac13.
\]
One may first differentiate a finite Gaussian series and then pass to
the limit using its summable covariance bound.  In particular,
\begin{equation}\label{lo:removal}
 0\leq a_n-A_n\leq \frac{s_n^2}{6}=o(n).
\end{equation}

\subsection{Restriction to a random subspace}

Fix $M\geq1$, $n\geq2$, and put $d=n+M$.  Our base measure $\nu_n$ is
the perturbed Gibbs measure on $S^{d-1}$ with
$P\sim\operatorname{Pois}(\alpha n)$ rows and perturbation
$s_dG_{d,\omega}$.  Denote its unnormalized density by $b(u)$ and its
partition function by $\widehat Z=\int b\,\dd\mu_d$.
For a Haar orthonormal $d\times n$ frame $U$, let
\[
 \widehat Z_U=\int_{S^{n-1}}b(Uz)\,\dd\mu_n(z),
 \qquad q(U)=\frac{\widehat Z_U}{\widehat Z}.
\]
Haar averaging followed by spherical averaging is uniform spherical
averaging in dimension $d$.  Thus $\int q(U)\,\dd U=1$ for every
realization of the base disorder.  The reverse relative entropy
\begin{equation}\label{lo:restrictionentropy}
 D_{n,M}=\E\int-\log q(U)\,\dd U
 =\E\log\widehat Z-\E_\omega a_n(s_d^2,\omega)\geq0
\end{equation}
is finite.  Conditional on each frame, the restricted rows are iid
standard Gaussians and the restricted perturbation has the same
coefficient vector $\omega$ and the same amplitude $s_d$.

For a random measure $\nu$ on a unit sphere, independent of the fresh
rows $g_j$, define
\begin{equation}\label{lo:rowfunctional}
 \Psi_\rho(\nu)=\E\log\int
   \exp\left\{-\beta\sum_{j=1}^{P_\rho}
                     \ind\{g_j\cdot u>0\}\right\}\,\dd\nu(u),
 \qquad P_\rho\sim\operatorname{Pois}(\rho).
\end{equation}
The expectation includes the randomness of $\nu$.  Jensen's inequality,
first in the spin and then in the fresh disorder, gives
\begin{equation}\label{lo:rowbounds}
 -\frac{\beta\rho}{2}\leq\Psi_\rho(\nu)
 \leq\rho\log\frac{1+e^{-\beta}}2.
\end{equation}
For the upper bound, condition on $P_\rho$ before applying Jensen.
Poisson superposition and \eqref{lo:restrictionentropy} give the exact
identity
\begin{equation}\label{lo:cavityidentity}
 a_{n+M}-a_n=D_{n,M}+\Psi_{\alpha M}(\nu_n)+\delta_{n,M},
 \qquad
 0\leq\delta_{n,M}\leq\frac{(n+M)^{2\gamma}-n^{2\gamma}}6=o(1).
\end{equation}
Both terms on the right use the same base measure $\nu_n$.  The error
is the difference between the two perturbation variances, as follows
directly from \eqref{lo:amplitude}.

To express $D_{n,M}$ through overlaps, take an independent standard
Gaussian $M\times d$ matrix $B$, which has full row rank almost surely.
Choose an orthonormal $d\times n$ frame $U$ for $\ker B$.  For $y\in\mathbb R^M$
with $y^{\mathsf T}(BB^{\mathsf T})^{-1}y<1$, put
\[
 t_y=B^{\mathsf T}(BB^{\mathsf T})^{-1}y,
 \qquad \rho_y=\sqrt{1-y^{\mathsf T}(BB^{\mathsf T})^{-1}y}.
\]
The spherical marginal density in $M$ orthonormal coordinates, followed
by the linear change of coordinates to $Bu$ for $u\sim\nu_n$, gives
the density
\[
 k_{\nu_n,B}(y)=
 \frac{\Gamma(d/2)\rho_y^{n-2}}
 {\pi^{M/2}\Gamma(n/2)\sqrt{\det(BB^{\mathsf T})}\,\widehat Z}
 \int_{S^{n-1}}b(t_y+\rho_yUz)\,\dd\mu_n(z).
\]
Here $\Gamma$ denotes Euler's gamma function.  This specifies a version
of the density throughout the interior of its ellipsoidal support.
Conditional on $B$, the projection of every base
row onto $\ker B$ is nonzero almost surely.  Since there are finitely
many rows, their hyperplanes intersect the kernel sphere in a set of
$\mu_n$-measure zero.  At every other $z$, all row signs in
$b(t_y+\rho_yUz)$ agree with those in $b(Uz)$ for sufficiently small
$y$, and the continuous perturbation converges to its value at $Uz$.
For each realization, $b$ is bounded on the sphere and
$0<\widehat Z<\infty$.  Dominated convergence in the displayed integral,
together with $\rho_y\to1$, therefore proves that this density is
continuous at zero almost surely.
Coarea at zero gives
\begin{equation}\label{lo:coarea}
 L_{n,0}:=(2\pi)^{M/2}k_{\nu_n,B}(0)
 =T_{d,M}(B)\,q(\ker B),\qquad
 T_{d,M}(B)=
 \frac{2^{M/2}\Gamma(d/2)}
 {\Gamma((d-M)/2)\sqrt{\det(BB^{\mathsf T})}}.
\end{equation}
The ratio $q$ depends only on the subspace, not its choice of frame.
The kernel is Haar and independent of the Gram matrix.  Gaussian
Gram--Schmidt represents $\det(BB^{\mathsf T})$ as the product of
independent chi-square variables of degrees $d,d-1,\ldots,d-M+1$.
For such a variable of degree $j$, its log expectation is
$\log2+\psi(j/2)=\log j+o(1)$, where $\psi$ is the logarithmic
derivative of $\Gamma$.  The fixed-$M$ gamma ratio in
\eqref{lo:coarea} consequently yields
\begin{equation}\label{lo:detcorrection}
 \E\log T_{d,M}\longrightarrow0,
 \qquad D_{n,M}=-\E\log L_{n,0}+\E\log T_{d,M}.
\end{equation}

For $\varepsilon>0$ introduce the smoothed likelihood
\[
 L_{n,\varepsilon}
 =\varepsilon^{-M}\int
          e^{-\norm{Bu}^2/(2\varepsilon^2)}\,\dd\nu_n(u).
\]
Gaussian integration gives
\begin{align}
 \E L_{n,\varepsilon}&=(1+\varepsilon^2)^{-M/2},
 \label{lo:firstmoment}\\
 \E L_{n,\varepsilon}L_{n,\eta}
 &=\E\left\langle
   \bigl((1+\varepsilon^2)(1+\eta^2)-R_{12}^2\bigr)^{-M/2}
                  \right\rangle .\label{lo:secondmoment}
\end{align}

\begin{lemma}[Overlap bound]\label{lo:cap}
There is $r_*<1$, depending only on $\alpha,\beta$, such that, for
every fixed $M$ and $L\geq0$,
\begin{equation}\label{lo:weightedtail}
 \E\left\langle(1-R_{12}^2)^{-L}
                    \ind\{|R_{12}|>r_*\}\right\rangle\longrightarrow0
\end{equation}
exponentially fast, uniformly in the coefficient vector $\omega$.
\end{lemma}

\begin{proof}
Let $\overline G=\int G_{d,\omega}\,\dd\mu_d$.  Since each Gaussian
row is positive on half the sphere, Jensen gives
$\widehat Z\geq\exp(-\beta P/2+s_d\overline G)$.
The density of two replicas relative to $\mu_d^{\otimes2}$ is thus
bounded by
\[
 \exp\{\beta P+s_d(G(u)+G(u')-2\overline G)\}.
\]
The variance of the centered Gaussian in parentheses is at most
$16C_\omega(1)\leq16/3$.  Averaging its exponential and the Poisson
count shows, for every nonnegative $h$,
\begin{equation}\label{lo:comparisonuniform}
 \E\langle h(R_{12})\rangle
 \leq \exp\{cn+\tfrac83s_d^2\}\E h(Q_d),
 \qquad c=\alpha(e^\beta-1),
\end{equation}
where $Q_d$ is the overlap of two uniform sphere points.  Its density
is $c_d(1-q^2)^{(d-3)/2}$ with $c_d=O(\sqrt d)$.
Choose $1-r_*^2=e^{-4(c+2)}$.  For sufficiently large $n$,
$\frac83s_d^2\leq n$ and $(d-3)/2-L\geq n/4$.
The expression in \eqref{lo:weightedtail} is therefore at most
$2c_d e^{(c+1)n}(1-r_*^2)^{n/4}\leq2c_d e^{-n}$.
\end{proof}

For large $n$, Lemma~\ref{lo:cap} makes
$\E\langle(1-R_{12}^2)^{-M/2}\rangle$ finite.  Formula
\eqref{lo:secondmoment} then shows that $L_{n,\varepsilon}$ is
Cauchy in $L^2$ as $\varepsilon\downarrow0$.  Density continuity
identifies this limit with $L_{n,0}$ in \eqref{lo:coarea}.
In particular $\E L_{n,0}=1$.  Expanding the squared difference by
\eqref{lo:secondmoment}, using uniform convergence on $[-r_*,r_*]$
and \eqref{lo:weightedtail} on its complement, gives
\begin{equation}\label{lo:uniformsmoothing}
 \lim_{\varepsilon\downarrow0}\limsup_{n\to\infty}
       \E\abs{L_{n,\varepsilon}-L_{n,0}}^2=0.
\end{equation}

\subsection{The limiting cavity functional}

For an overlap law $\mathscr L$ with $|R_{12}|\leq r<1$, conditionally
on its array take $M$ independent centered Gaussian sequences
$(X_j^\ell)_{\ell\geq1}$ of covariance $R$.  Define
\begin{equation}\label{lo:lambdadef}
 \Lambda_{M,\varepsilon}(\mathscr L)
 =L^2\!\operatorname{-lim}_{K\to\infty}
       \frac1K\sum_{\ell=1}^K\varepsilon^{-M}
       \exp\left\{-\frac1{2\varepsilon^2}
                         \sum_{j=1}^M(X_j^\ell)^2\right\}.
\end{equation}
The summands are bounded and exchangeable.  Expanding the square of
two empirical averages proves the $L^2$ Cauchy property: terms with
distinct replica indices cancel, and the remaining contribution is
$O(K^{-1}+K'^{-1})$.  Its moments are
\begin{equation}\label{lo:detmoments}
 \E\Lambda_{M,\varepsilon}^{\,b}
 =\E\det(R_{[b]}+\varepsilon^2 I_b)^{-M/2},
 \qquad b\geq1.
\end{equation}
Here $R_{[b]}=(R_{\ell\ell'})_{1\leq\ell,\ell'\leq b}$ and $I_b$
is the $b\times b$ identity matrix. The expectation on the right of
\eqref{lo:detmoments} is over the overlap-array law $\mathscr L$.
The cross second moment is \eqref{lo:secondmoment} with expectation
under $\mathscr L$.  The overlap cap hence permits the definition
\[
 \Lambda_M(\mathscr L)
 =L^2\!\operatorname{-lim}_{\varepsilon\downarrow0}
                          \Lambda_{M,\varepsilon}(\mathscr L),
 \qquad \E\Lambda_M=1.
\]

Fresh-row partitions have the parallel definition
\begin{equation}\label{lo:freshlimit}
 Z_b(\mathscr L)=L^2\!\operatorname{-lim}_{K\to\infty}
 \frac1K\sum_{\ell=1}^K
     \exp\left\{-\beta\sum_{j=1}^b\ind\{X_j^\ell>0\}\right\},
 \qquad
 \Psi_\rho(\mathscr L)=\E\log Z_{P_\rho}(\mathscr L).
\end{equation}
Here $b$ Gaussian sequences are used and $Z_0=1$.  Since
$e^{-\beta b}\leq Z_b\leq1$, all these logarithms are integrable
and \eqref{lo:rowbounds} holds.  Set
\begin{equation}\label{lo:functional}
 \mathcal F_M(\mathscr L)
 =-\E\log\Lambda_M(\mathscr L)+\Psi_{\alpha M}(\mathscr L).
\end{equation}
Initially the first term is understood as an extended real number.

\begin{lemma}[Cavity limit]\label{lo:cavitylimit}
Suppose the averaged overlap laws of the base measures $\nu_n$
converge along a subsequence to $\mathscr L$.  Then
\begin{equation}\label{lo:subsequential}
 \liminf_n(a_{n+M}-a_n)\geq\mathcal F_M(\mathscr L)
\end{equation}
along that subsequence.
\end{lemma}

\begin{proof}
Lemma~\ref{lo:cap} gives $|R_{12}|\leq r_*$ under $\mathscr L$.
For each positive $\varepsilon$, \eqref{lo:detmoments} and its
finite-dimensional analogue imply convergence of every moment of
$L_{n,\varepsilon}$ to that of $\Lambda_{M,\varepsilon}$.
The variables lie in $[0,\varepsilon^{-M}]$, so moment convergence
implies convergence in distribution.  Formula
\eqref{lo:uniformsmoothing} and its capped-law counterpart imply
\begin{equation}\label{lo:densityweak}
 L_{n,0}\ \Longrightarrow\ \Lambda_M(\mathscr L).
\end{equation}
Write $\log^+v=\max\{\log v,0\}$ and
$\log^-v=\max\{-\log v,0\}$, allowing $\log^-0=+\infty$.
For any nonnegative mean-one variable $V$, integration of the Markov
tail gives
\begin{equation}\label{lo:logpositive}
 \E[\log^+V\,\ind\{\log^+V>t\}]\leq(t+1)e^{-t}.
\end{equation}
Thus the positive logarithms in \eqref{lo:densityweak} are uniformly
integrable.  Fatou's lemma for the nonnegative negative-logarithm
part gives the signed inequality
\begin{equation}\label{lo:logfatou}
 \liminf_n[-\E\log L_{n,0}]
 \geq-\E\log\Lambda_M(\mathscr L).
\end{equation}

For a fixed fresh-row count $b$, every moment of its partition is a
bounded Gaussian sign integral depending on a finite covariance
matrix.  These integrals are continuous even at singular matrices:
covariance convergence gives convergence of the Gaussian vectors,
and their sign discontinuities have probability zero because each
coordinate has variance one.  Moment convergence on
$[e^{-\beta b},1]$ therefore gives convergence of the expected
logarithm.  The bound $|\log Z_b|\leq\beta b$ allows uniform
truncation of the Poisson count and proves
$\Psi_{\alpha M}(\nu_n)\to\Psi_{\alpha M}(\mathscr L)$.
Now combine \eqref{lo:cavityidentity}, \eqref{lo:detcorrection},
and \eqref{lo:logfatou}.
\end{proof}

\begin{lemma}[Continuity under an overlap cap]\label{lo:continuity}
On the set of overlap laws with $|R_{12}|\leq r<1$,
$\mathcal F_M$ is finite, continuous, and affine under mixtures of
array laws.  Moreover,
\begin{equation}\label{lo:lognegative}
 \E\log^-\Lambda_M\leq\frac{Mr}{2(1-r)},\qquad
 \E(\log^-\Lambda_M)^2
       \leq\frac{M(M+2)r^2}{4(1-r)^2}.
\end{equation}
\end{lemma}

\begin{proof}
Use the directing measure $\nu$ from \eqref{rpc:DS}.  The upper
overlap bound implies $\norm h^2\leq r$ for $\nu$-almost every $h$,
almost surely.  Otherwise, some set $\{\norm h^2>r+\delta\}$ would
have positive mass.  Separability would give a sufficiently small
ball of positive mass inside this set, and every pair in that ball
would have inner product greater than $r$, a contradiction.

Let $G_1,\ldots,G_M$ be independent centered Gaussian fields with
$\E G_j(h)G_j(h')=h\cdot h'$. Put $v(h)=1-\norm h^2$.  In the Gaussian sequences defining
\eqref{lo:lambdadef}, the residual variance at $h$ is $v(h)$.
The conditional law of large numbers and Gaussian integration give
\begin{align}
 \Lambda_{M,\varepsilon}
 &=\int(v(h)+\varepsilon^2)^{-M/2}
       \exp\left\{-\frac{\sum_jG_j(h)^2}
                          {2(v(h)+\varepsilon^2)}\right\}\,\dd\nu(h),
 \nonumber\\
 \Lambda_M
 &=\int v(h)^{-M/2}
       \exp\left\{-\frac{\sum_jG_j(h)^2}{2v(h)}\right\}\,\dd\nu(h)>0.
 \label{lo:directingdensity}
\end{align}
The fields can be defined $\nu$-almost everywhere through their
Gaussian series.  Since $v(h)\geq1-r$, dominated convergence identifies
the second expression with the previously defined $L^2$ limit.
Writing $S=\int\sum_jG_j(h)^2\,\dd\nu(h)$, Jensen gives
$\log\Lambda_M\geq-S/[2(1-r)]$.  Also
\[
 \E S\leq Mr,\qquad
 \E S^2\leq\E\int\Bigl(\sum_jG_j(h)^2\Bigr)^2\,\dd\nu(h)
                 \leq M(M+2)r^2,
\]
which proves \eqref{lo:lognegative}.

At fixed $\varepsilon>0$ the moment kernels
\eqref{lo:detmoments} are continuous.  The cross second-moment kernel
in \eqref{lo:secondmoment} converges uniformly down to zero noise on
$|R_{12}|\leq r$.  Thus the law of $\Lambda_M$ is continuous on this
class.  Equations \eqref{lo:lognegative} and \eqref{lo:logpositive}
give uniform integrability of its logarithm, and hence continuity of
its expected logarithm.  The fresh-row continuity argument in
Lemma~\ref{lo:cavitylimit} applies to any convergent array laws.
Finally, sampling a component law first and then its array and
Gaussian sequences realizes a mixture of the laws of both
$\Lambda_M$ and $Z_b$.  Their integrable log expectations therefore
average linearly, proving affinity.
\end{proof}

\subsection{Ghirlanda--Guerra identities with the row disorder}

We now show that the possible limiting laws in the cavity inequality
are mixtures of laws satisfying \eqref{rpc:GG}.  All disorder means
in this argument are taken at a fixed coefficient vector $\omega$.
The concentration calculation includes the Poisson rows.

Fix a power $p$, keep all coefficients except $\omega_p$ fixed, and
write $v_p=s_d2^{-p}$, $t=\omega_p$, $\theta(t)=\log\widehat Z$,
and $\psi(t)=\E\theta(t)$.  Uniformly for $t\in[-1,2]$ and the other
coefficients in $[0,1]$,
\begin{equation}\label{lo:concentration}
 \E|\theta(t)-\psi(t)|\leq B_n,
 \qquad B_n=C_{\alpha,\beta}(\sqrt n+s_d+1).
\end{equation}
Indeed, conditional on the row count, replacing one row changes
$\log\widehat Z$ by at most $\beta$, so the row variance is
$O(\beta^2P)$.  Its conditional mean is $\beta$-Lipschitz in $P$;
comparison with an independent copy of $P$ bounds the resulting
variance by $\beta^2\operatorname{Var}P$.  Gaussian Poincare bounds
the perturbation variance by $Cs_d^2$, because the squared Gaussian
Lipschitz norm of the log partition is at most
$s_d^2\sup_u\operatorname{Var}G_{d,\omega}(u)$.
Variance decomposition proves \eqref{lo:concentration}.

Differentiation and Gaussian integration by parts give
\begin{equation}\label{lo:derivatives}
 \theta'=v_p\langle g_{p,d}\rangle,\qquad
 \psi'=v_p^2t\bigl(1-\E\langle R_{12}^p\rangle\bigr),\qquad
 \psi''=v_p^2\E\langle(g_{p,d}-\langle g_{p,d}\rangle)^2\rangle.
\end{equation}
Since $0\leq\E\langle R_{12}^p\rangle\leq1$, integration of the last
identity on $[0,1]$ and Cauchy--Schwarz imply
\begin{equation}\label{lo:thermalfluctuation}
 \int_0^1\E\langle|g_{p,d}-\langle g_{p,d}\rangle|\rangle\,\dd t
 \leq1.
\end{equation}
For $0<h\leq1$, the two secant bounds for a convex derivative give
\[
 \E|\theta'(t)-\psi'(t)|
 \leq\psi'(t+h)-\psi'(t-h)+\frac{3B_n}{h}.
\]
Integrating in $t\in[0,1]$ and using
$|\psi'(t)|\leq2v_p^2$ on $[-1,2]$ yields
\[
 \int_0^1\E|\langle g_{p,d}\rangle
                     -\E\langle g_{p,d}\rangle|\,\dd t
 \leq C\left(hv_p+\frac{B_n}{hv_p}\right).
\]
For fixed $p$ and large $n$, choose $h=\sqrt{B_n}/v_p\leq1$.
Together with \eqref{lo:thermalfluctuation}, this proves
\begin{equation}\label{lo:totalfluctuation}
 \E_\omega\E\langle|g_{p,d}-\E\langle g_{p,d}\rangle|\rangle
 \leq1+C\sqrt{B_n}.
\end{equation}

Let $|f|\leq1$ depend on the overlaps of $r$ replicas.  Gaussian
integration by parts in $g_{p,d}$, subtracting the same formula for
$\E\langle g_{p,d}\rangle$, bounds the absolute defect in
\eqref{rpc:GG} with $\psi(q)=q^p$ by
\[
 \frac{\E\langle|g_{p,d}-\E\langle g_{p,d}\rangle|\rangle}
      {r v_p\omega_p},\qquad \omega_p>0.
\]
For $\omega_p\leq\eta$ bound the defect by two, and integrate the
preceding estimate on $[\eta,1]$.  Consequently
\begin{equation}\label{lo:GGdefect}
 \E_\omega\Delta_{r,p,f}(\omega)
 \leq2\eta+\frac{C(1+\sqrt{B_n})}{\eta v_p}
 \longrightarrow0
\end{equation}
by choosing $\eta\downarrow0$ sufficiently slowly.  This is possible
because $(1+\sqrt{B_n})/v_p=O(n^{1/4-\gamma})\to0$.
The calculation is the convexity perturbation argument of
\cite[Section 9]{Panchenko2014SK}; the estimates above give it here
with the full row disorder and coefficients uniform on $[0,1]$.

For each $\omega$, let $\mathscr L_{n,\omega}$ be the overlap-array
law after averaging the row and Gaussian disorder at this fixed
coefficient vector, and let $\mathfrak m_n$ be the distribution of
$\mathscr L_{n,\omega}$ when $\omega$ is sampled.  Probability laws
on the compact array-law space form a compact space.  Along any
subsequence we may therefore pass to a further subsequence on which
\[
 \mathfrak m_n\Longrightarrow\mathfrak m,
 \qquad
 \E_\omega\mathscr L_{n,\omega}
       \Longrightarrow\overline{\mathscr L}
       :=\int\mathscr L\,\dd\mathfrak m(\mathscr L).
\]
Positive semidefiniteness, exchangeability, and diagonal one are
closed conditions.  Lemma~\ref{lo:cap}, tested against continuous
functions on forbidden overlap intervals, implies $|R_{12}|\leq r_*$
under $\mathfrak m$-almost every law.
For continuous cylinder tests $f$, the absolute defect
$\Delta_{r,p,f}$ is a continuous function of the law $\mathscr L$:
its product term is the product of two scalar integrals under that
same law.  Formula \eqref{lo:GGdefect}, first for a countable dense
collection of tests, therefore implies \eqref{rpc:GG} for
$\mathfrak m$-almost every component.  Uniform approximation and
the monotone-class theorem give all bounded measurable tests.
This keeps the product of conditional expectations in the identities;
the barycenter need not itself satisfy them.

Each such component also has nonnegative overlaps, by the positivity
consequence of the exact identities
\cite[Theorem 7, p.~265]{Panchenko2014SK}.  Here is the argument for
our unit-diagonal arrays.  Fix a component $\mathscr L$, let
$a=\mathscr L(R_{12}\leq-\varepsilon)$ for $\varepsilon>0$, and let
$A_r$ be the event that $R_{ij}\leq-\varepsilon$ for every
$1\leq i<j\leq r$.  On $A_r$, all the old overlaps belong to this
interval, so \eqref{rpc:GG} and exchangeability give
\[
 \mathscr L\bigl(A_r\cap\{R_{i,r+1}>-\varepsilon\}\bigr)
 =\frac{1-a}{r}\mathscr L(A_r),\qquad 1\leq i\leq r.
\]
A union bound gives $\mathscr L(A_{r+1})\geq a\mathscr L(A_r)$,
and hence $\mathscr L(A_r)\geq a^{r-1}$.  But positive semidefiniteness
on $A_r$ would imply
$0\leq\sum_{i,j\leq r}R_{ij}\leq r-r(r-1)\varepsilon$, which is
impossible for sufficiently large $r$.  Thus $a=0$ for every
$\varepsilon>0$, and $0\leq R_{12}\leq r_*$ in almost every component.
Since every subsequential limit has this property, compactness and
the Portmanteau theorem also give, for every $\varepsilon>0$,
\begin{equation}\label{lo:positivity}
 \E_\omega\E\langle\ind\{R_{12}\leq-\varepsilon\}\rangle
 \longrightarrow0.
\end{equation}

\subsection{Evaluation on cascades and completion of the proof}

For the finite cascade with profile $(q,x)$, write $v=1-q_k$ and let
$z_{j,a}$, $1\leq j\leq M$, be independent ancestral Gaussian fields
of covariance $Q$.  Formula \eqref{rpc:residual} and Gaussian
integration in \eqref{lo:lambdadef} give
\begin{equation}\label{lo:RPClambda}
 \Lambda_M=v^{-M/2}\sum_aw_a
           \exp\left\{-\frac1{2v}\sum_{j=1}^M z_{j,a}^2\right\}.
\end{equation}
Put $\lambda_i=v+\sum_{j=i+1}^k x_j(q_j-q_{j-1})$, so
$\lambda_k=v$.  The Gaussian transform at level $i$ satisfies
\[
 \frac1{x_i}\log\E_Z
 \exp\left\{-\frac{x_i(h+\sqrt{q_i-q_{i-1}}Z)^2}{2\lambda_i}\right\}
 =-\frac1{2x_i}\log\frac{\lambda_{i-1}}{\lambda_i}
   -\frac{h^2}{2\lambda_{i-1}}.
\]
The $M$ coordinates factorize at each step of
Lemma~\ref{rpc:marking}.  Averaging the root Gaussian of variance
$q_0$ yields
\begin{equation}\label{lo:RPCentropy}
 -\frac1M\E\log\Lambda_M
 =\frac12\log v+\frac{q_0}{2\lambda_0}
    +\frac12\sum_{i=1}^k\frac1{x_i}
                         \log\frac{\lambda_{i-1}}{\lambda_i}
 =\mathcal E(q,x).
\end{equation}
The marks in \eqref{lo:RPClambda} are bounded above.  The displayed
Gaussian integral gives strictly positive, finite conditional moments,
and all resulting logarithms have at most quadratic growth in the
parent field.  Thus the marking lemma's integrability requirements
hold.

For one fresh row, integrating its independent residual Gaussian gives
the terminal function \eqref{eq:terminal}. The marking recursion is
exactly \eqref{eq:recursion}, whose root average is
$\mathcal J_\beta(q,x)$.
For $b$ fresh rows the terminal log mark is a sum of $b$ such
functions, in independent Gaussian coordinates.  Every step of
the recursion preserves this separability.  Its expected log
partition is consequently $b\mathcal J_\beta(q,x)$, and averaging
the Poisson count proves
\begin{equation}\label{lo:RPCvalue}
 \mathcal F_M(\mathscr L_{q,x})
 =M\mathcal E(q,x)+\alpha M\mathcal J_\beta(q,x)
 =M\mathcal P_{\alpha,\beta}(q,x).
\end{equation}
In particular $\mathcal E\geq0$, since $\E\Lambda_M=1$, and
\eqref{lo:rowbounds} and the trial $k=0$, $q_0=0$ give
\begin{equation}\label{lo:Pfinite}
 -\frac{\alpha\beta}{2}\leq P_*(\alpha,\beta)
 \leq\alpha\log\frac{1+e^{-\beta}}2.
\end{equation}

\begin{proof}[Proof of Theorem~\ref{lo:pressure}]
Fix $M\geq1$ and take a subsequence attaining the lower limit of
$a_{n+M}-a_n$.  The increments are bounded below by
$-\alpha\beta M/2$ by \eqref{lo:cavityidentity}.  If their lower
limit is $+\infty$ the desired increment bound is immediate;
otherwise extract a further subsequence with the law-valued limit
$\mathfrak m$ just constructed.  Lemmas~\ref{lo:cavitylimit} and
\ref{lo:continuity} imply
\[
 \liminf_n(a_{n+M}-a_n)
 \geq\mathcal F_M(\overline{\mathscr L})
 =\int\mathcal F_M(\mathscr L)\,\dd\mathfrak m(\mathscr L).
\]
For almost every component, Proposition~\ref{rpc:representation}
provides finite cascade laws $\mathscr L_j$ converging to
$\mathscr L$, all with the same cap $r_*$.  Therefore
\[
 \mathcal F_M(\mathscr L)
 =\lim_j\mathcal F_M(\mathscr L_j)
 \geq M P_*(\alpha,\beta)
\]
by continuity and \eqref{lo:RPCvalue}.  The component functionals
are integrable by \eqref{lo:lognegative} and \eqref{lo:rowbounds}.
This proves, for every fixed $M$,
\begin{equation}\label{lo:incrementbound}
 \liminf_{n\to\infty}(a_{n+M}-a_n)
 \geq M P_*(\alpha,\beta).
\end{equation}
For any $\varepsilon>0$, all sufficiently large increments in
\eqref{lo:incrementbound} are at least
$M(P_*-\varepsilon)$.  Telescope on each of the finitely many
residue classes modulo $M$ and divide by $n$.  It follows that
$\liminf_n a_n/n\geq P_*$.  Equation \eqref{lo:removal}
removes the perturbation after normalization.

Finally couple $P_n$ and the deterministic $m_n$ using one infinite
sequence of Gaussian rows.  Each added or deleted row changes the
log partition by at most $\beta$, whence
\begin{equation}\label{lo:poissontransfer}
 \left|\E\log Z_{n,m_n}(\beta)-A_n\right|
 \leq\beta\E|P_n-m_n|
 \leq\beta\bigl(\sqrt{\alpha n}+|m_n-\alpha n|\bigr)=o(n).
\end{equation}
Dividing by $n$ proves the theorem.
\end{proof}

\section{The pressure upper bound}\label{up:section}

\begin{theorem}\label{up:pressure}
For every $\beta>0$, every finite profile $(q,x)$, and every sequence
$m_n/n\to\alpha<\infty$,
\[
 \limsup_{n\to\infty}\Phi_{n,m_n}(\beta)
 \leq \mathcal P_{\alpha,\beta}(q,x).
\]
Consequently $\limsup_n\Phi_{n,m_n}(\beta)\leq P_*(\alpha,\beta)$.
\end{theorem}

Fix the profile throughout this section. Set $v=1-q_k$,
$\pi_i=x_{i+1}-x_i$, and $\alpha_n=m/n$. We use spins
$\sigma=\sqrt n\,u$ and normalized surface measure $\mu_n$ on
$\mathbb S_n=\{\sigma\in\mathbb R^n:\|\sigma\|^2=n\}$.
Here $\mu_n$ denotes the pushforward of the unit-sphere measure under
$u\mapsto\sqrt n\,u$.
For replicas indexed by $a,b$, write $R_{ab}=n^{-1}\sigma^a\cdot\sigma^b$.
Let $(w_\gamma)$ be the cascade of Definition~\ref{rpc:construction}.
Its tree kernel is $Q_{\gamma\gamma'}=q_i$ when the common ancestor
has depth $i$; in particular $Q_{\gamma\gamma}=q_k$.

\subsection{Interpolation and Gaussian perturbation}

Let $z_i(\gamma)$, $1\leq i\leq m$, be independent Gaussian tree
fields with covariance $Q$, independent of the matrix $W$ and the cascade.
For a deterministic covariance profile
$0\leq a_0\leq\cdots\leq a_k$, let $g_j(\gamma)$,
$1\leq j\leq n$, be independent tree fields with covariance
$a(Q_{\gamma\gamma'})=a_i$ at depth $i$. These fields are independent
of the rows. Write $W_i$ for the $i$th row of $W$. For the
interpolation parameter $0\leq t\leq1$, define
\begin{align}
 X_i(t,\sigma,\gamma)
 &=\sqrt t\,\frac{W_i\cdot\sigma}{\sqrt n}
       +\sqrt{1-t}\,z_i(\gamma),\nonumber\\
 U_t(h)&=\log\E\exp\{-\beta\ind\{h+\sqrt{(1-t)v}\,Z>0\}\},
       \qquad \chi_t=\partial_hU_t,\label{up:channel}\\
 F_n(t,a)&=\frac1n\E\log\sum_\gamma w_\gamma
   \int_{\mathbb S_n}
   \exp\left\{\sum_{i=1}^m U_t(X_i)
                 +\sum_{j=1}^n g_j(\gamma)\sigma_j\right\}
       \mu_n(\dd\sigma).\label{up:interpolation}
\end{align}
The score $\chi_t$ is defined for $t<1$, where $U_t$ is smooth.
Here and below $Z$ denotes a standard normal variable in a one-dimensional
Gaussian expectation. Put
\[
 \mathcal B(q,a)=a_k-\sum_{i=0}^k\pi_iq_i a_i,
 \qquad V_n(t,a)=F_n(t,a)-\tfrac12\mathcal B(q,a).
\]
The endpoint $F_n(1,0)$ is the original pressure.

Fix $\varepsilon>0$ and $A<\infty$, and restrict for now to
$0\leq t\leq1-\varepsilon$ and $m\leq An$.
All spatial derivatives of $U_t$ of a fixed order are bounded on this
interval, uniformly in $h$. Define the score overlap
\begin{equation}\label{up:score}
 C_{ab}=\frac1n\sum_{i=1}^m\chi_t(X_i^a)\chi_t(X_i^b),
 \qquad \widetilde C_{ab}=C_{ab}/B_\varepsilon,
 \qquad B_\varepsilon>1+A\sup_{t,h}|\chi_t(h)|^2.
\end{equation}
The three kernels $R,\widetilde C,Q$ are positive semidefinite, with
diagonals at most one. Since $U_t$ is nonincreasing, $C_{ab}\geq0$.

Enumerate all kernels
\begin{equation}\label{up:kernels}
 K_\ell=(\vartheta_{\ell,R}R+
          \vartheta_{\ell,C}\widetilde C+
          \vartheta_{\ell,Q}Q)^{p_\ell},
\end{equation}
where $p_\ell\geq1$ and the three nonnegative rational weights sum to
one; powers are entrywise. Choose $\tau_\ell>0$ so that
$\sum_\ell\tau_\ell^2(1+p_\ell)^4<\infty$, and let
$\boldsymbol u=(u_\ell)_{\ell\geq1}$ have independent uniform
$[0,1]$ coordinates. At fixed $\boldsymbol u$, set
\[
 D_{\boldsymbol u}=\sum_\ell\tau_\ell^2u_\ell^2K_\ell,
 \qquad s_n=n^\gamma,\quad \tfrac14<\gamma<\tfrac12.
\]
Conditional on all row disorder, comprising $W$ and the tree fields
$z_i$, let $G_{\boldsymbol u}$ be a Gaussian field with covariance
$D_{\boldsymbol u}$ and add
\begin{equation}\label{up:perturbation}
 s_nG_{\boldsymbol u}(\sigma,\gamma)
       -\frac{s_n^2}{2}D_{\boldsymbol u}
          ((\sigma,\gamma),(\sigma,\gamma))
\end{equation}
to the exponent in \eqref{up:interpolation}. Denote its expected
normalized log partition by $F_{n,\boldsymbol u}(t,a)$ and its
coefficient average by $\overline F_n(t,a)$; put
$\overline V_n=\overline F_n-\mathcal B/2$.
At fixed coefficients, $\E$ averages all other disorder.

We specify a realization of this Gaussian field that preserves the
cascade marking structure. Expanding \eqref{up:kernels} gives positive
linear combinations of $R^r\widetilde C^sQ^h$.
The first two factors are tensor Gram kernels of the spin and score
vectors. For $h\geq1$, decompose $Q^h$ into its nonnegative increments
along the tree and place independent Gaussian tensor functions at the
corresponding nodes. The term $h=0$ is a root function. A leaf evaluates
only functions on its ancestral path, at its own spin and score vector.
All Gaussian coefficients are independent of the Poisson points defining
the cascade. A root function can have descendant scores as arguments:
conditional on the function and the other root decorations, different
child subtrees are still independent and identically distributed.
The diagonal term in \eqref{up:perturbation} has the same property.
Thus Lemma~\ref{rpc:marking} applies to these marks.

For completeness, their required moments are finite. On each fixed
leaf, the spin and score vectors range over a compact finite-dimensional
set. The tensor estimate
$\|z^{\otimes p}-y^{\otimes p}\|\leq p\|z-y\|$ for
$\|z\|,\|y\|\leq1$, together with the summability imposed above,
bounds the canonical Gaussian metric by a constant times the Euclidean
metric. The field therefore has a continuous version with finite
exponential moments of its supremum on this set. The row potential is
bounded, and the strict inequalities between the cascade masses give
the lower stable moments required at preceding levels.

Two consequences will be used repeatedly. First, Gaussian variance
interpolation with the diagonal subtraction gives
\begin{equation}\label{up:cost}
 -C\frac{s_n^2}{n}\leq F_{n,\boldsymbol u}(t,a)-F_n(t,a)\leq0.
\end{equation}
Indeed, the derivative of the unnormalized expected log partition when
the added covariance is $rD_{\boldsymbol u}$ is
$-\E\langle D_{\boldsymbol u,12}\rangle/2$.
The bracket is the squared norm of the mean kernel feature and belongs
to $[0,\sup D_{\boldsymbol u,11}]$. This proof is conditional on the
rows and is uniform in $a,t,\boldsymbol u$.

Second, at every fixed coefficient vector the perturbed Gibbs measure
satisfies
\begin{equation}\label{up:Qlaw}
 \E\langle\ind\{Q_{12}=q_i\}\rangle=\pi_i.
\end{equation}
To prove this, add an independent scalar tree field with covariance
$rH(Q)$, where $H(q_i)$ is nonnegative and nondecreasing.
The marking recursion gives the exact change
$\frac r2\sum_{i=1}^k x_i[H(q_i)-H(q_{i-1})]$ in the expected log
partition. The root increment has mean zero. Gaussian differentiation
gives the same change to first order as
$\frac r2[H(q_k)-\E\langle H(Q_{12})\rangle]$.
Summation by parts and the successive threshold choices of $H$ prove
\eqref{up:Qlaw}. The argument also holds conditional on the original
common root decorations.

\subsection{Concentration and joint overlap identities}

\begin{lemma}\label{up:concentration}
For $a_k\leq L$, the perturbed unnormalized log partition $\Theta_n$
satisfies
\begin{equation}\label{up:variance}
 \operatorname{Var}(\Theta_n)
 \leq C\left(n+s_n^2+\frac{s_n^4}{n}+1\right).
\end{equation}
The constant is uniform in $0\leq t\leq1-\varepsilon$, $m\leq An$,
and $\boldsymbol u$. The same estimate holds if one coefficient is
allowed to range over $[-1,2]$.
\end{lemma}

\begin{proof}
Let $\mathcal F_{\rm row}$ be generated by $W$ and the root row Gaussians,
let $\mathcal F_{\rm spin}$ also include the root spin Gaussian vector,
and let $\mathcal F_{\rm root}$ also include the root tensor functions.
These sigma-fields are nested. With
$M_a=\E[\Theta_n\mid\mathcal F_a]$ for
$a\in\{{\rm row},{\rm spin},{\rm root}\}$, total variance gives
\begin{align*}
 \Var(\Theta_n)
 &=\E\Var(\Theta_n\mid\mathcal F_{\rm root})
   +\E\Var(M_{\rm root}\mid\mathcal F_{\rm spin})\\
 &\quad+\E\Var(M_{\rm spin}\mid\mathcal F_{\rm row})
   +\Var(M_{\rm row}).
\end{align*}
We bound the four terms in this order. If $k\geq1$, write
the normalized cascade partition as $T/S$, where $T$ and $S$ are its
marked and unmarked root sums. Conditional on the root decorations,
each has a positive $x_1$-stable law with a deterministic scale. This
follows by Poisson marking; the inner subtree totals have finite
$x_1$-moments since the next mass is larger. Independence of $T$ and
$S$ is unnecessary. If $S_x$ has Laplace transform $e^{-t^x}$, then
\[
 \E S_x^z=\frac{\Gamma(1-z/x)}{\Gamma(1-z)},\qquad z<x,
 \qquad
 \operatorname{Var}(\log S_x)=\frac{\pi^2}{6}(x^{-2}-1).
\]
The moment identity follows by integrating the Laplace transform
(and its difference from one for positive $z$); it justifies
differentiation near zero. Hence
\[
 \operatorname{Var}(\log(T/S)\mid\text{root})
       \leq4\operatorname{Var}(\log S_{x_1}).
\]
For $k=0$ this conditional variance is zero.

The squared Lipschitz constant of $M_{\rm root}$ in the root Gaussian tensor coefficients
is at most $Cs_n^2$, since all kernel diagonals are bounded.
After their average, the squared Lipschitz constant of $M_{\rm spin}$ in the
standard Gaussian coordinates of the root spin vector is at most $na_0$.
Gaussian Poincar\'e inequalities
bound these two variance contributions by $Cs_n^2+nL$.

For the last term, $M_{\rm row}$ is a function only of the independent
root row tuples $(W_i,z_i(\varnothing))$, where $\varnothing$ denotes
the tree root.
Couple all descendant row increments identically before and after a
replacement, and fix them, the cascade weights, and the entire spin
field. Write $H,H'$ for the two base exponents and $D,D'$ for their
feature covariances. Only the $i$th row changes, so
$\|H-H'\|_\infty\leq\beta$ and
\[
 \|\widetilde C-\widetilde C'\|_\infty
 \leq\frac{2\sup_{t,h}|\chi_t(h)|^2}{nB_\varepsilon},\qquad
 \|D-D'\|_\infty\leq\frac Cn.
\]
Here the norms are over all states or state pairs, respectively; the
second bound follows from $|x^p-y^p|\leq p|x-y|$ on $[-1,1]$ and
$\sum_\ell\tau_\ell^2u_\ell^2p_\ell<\infty$.
Let $\mathcal P(H,D)$ be the conditional expected log partition,
averaging all feature Gaussians while keeping these decorations fixed.
Interpolate the positive semidefinite covariance
$D_r=(1-r)D+rD'$, with compensator $-s_n^2D_{r,11}/2$ and base
exponent $H$. Conditional Gaussian differentiation gives
\[
 \frac{\dd}{\dd r}\mathcal P(H,D_r)
 =-\frac{s_n^2}{2}\E_g\langle D'_{12}-D_{12}\rangle_r,
 \qquad
 |\mathcal P(H,D')-\mathcal P(H,D)|\leq Cs_n^2/n.
\]
Changing $H$ to $H'$ at covariance $D'$ costs at most $\beta$
pointwise. The bounds are independent of the fixed descendant
increments, weights, and spin field, so averaging those decorations
gives $|M_{\rm row}-M'_{\rm row}|\leq\beta+Cs_n^2/n$.
The calculation can first be made with finite features, leaves and
sphere quadratures; Lemma~\ref{ud:limits} justifies their removal.
Bounded differences for the $m$ root tuples give
$Cm(\beta+Cs_n^2/n)^2$. Combining the conditional variance bounds proves
\eqref{up:variance}.
\end{proof}

In particular, uniformly under these conditions,
\begin{equation}\label{up:fluctuation}
 \E|\Theta_n-\E\Theta_n|\leq B_n,\qquad B_n=C\sqrt n.
\end{equation}
We next derive the full disorder overlap identities, including the
variable diagonal of the score kernel.

Fix $K=K_\ell$ and all coefficients except $u=u_\ell$, and write
$b=s_n\tau_\ell$. Let $g_K$ be its unit Gaussian field,
$\theta(u)=\Theta_n$, and $\psi(u)=\E\theta(u)$. Set
$h_u=g_K-buK_{11}$. Differentiation and Gaussian integration by parts give
\begin{equation}\label{up:GGderivatives}
 \theta'=b\langle h_u\rangle,\qquad
 \psi'=-b^2u\E\langle K_{12}\rangle,\qquad
 \theta''=b^2\operatorname{Var}_{\langle\cdot\rangle}(h_u)
                      -b^2\langle K_{11}\rangle.
\end{equation}
Thus $\theta(u)+b^2u^2/2$ is convex. Integrating the last identity and
using the middle one at $u=0,1$ yields
\[
 \int_0^1\E\operatorname{Var}_{\langle\cdot\rangle}(h_u)\,\dd u
 =-\E\langle K_{12}\rangle_{u=1}
       +\int_0^1\E\langle K_{11}\rangle\,\dd u\leq1.
\]
Here $\E\langle K_{12}\rangle\geq0$ by the Gram representation.

We include the convex derivative estimate. For a random convex function
$f$ with $\E|f-\E f|\leq B_n$ on $[-1,2]$, forward and backward
secants give, for $0<\delta\leq1$,
\[
 \E|f'(u)-(\E f)'(u)|
 \leq (\E f)'(u+\delta)-(\E f)'(u-\delta)+3B_n/\delta.
\]
Apply this to $f=\theta+b^2u^2/2$. Its expected derivative is bounded
by $Cb^2$ on $[-1,2]$. After integration over $u\in[0,1]$ and division
by $b$, the right side is at most
$C(\delta b+B_n/(\delta b))$.
For fixed $\ell$, the choice $\delta=\sqrt{B_n}/b$ is admissible
eventually because $\gamma>1/4$. The thermal variance bound and
Cauchy--Schwarz consequently give
\begin{equation}\label{up:GGfluctuation}
 \int_0^1\E\langle|h_u-\E\langle h_u\rangle|\rangle\,\dd u
       \leq C(1+\sqrt{B_n}).
\end{equation}

If $f$ is a bounded function of the off-diagonal vector overlaps of
$r$ replicas, Gaussian integration by parts cancels all self-overlaps:
\[
 \E\langle f h_u^1\rangle
 =bu\,\E\left\langle f\left(\sum_{a=2}^rK_{1a}
                                      -rK_{1,r+1}\right)\right\rangle.
\]
Subtracting the corresponding mean and using \eqref{up:GGfluctuation}
shows, for $\|f\|_\infty\leq1$, that the coefficient-averaged absolute
defect in
\begin{equation}\label{up:GG}
 \E\langle fK_{1,r+1}\rangle
 =\frac1r\E\langle f\rangle\E\langle K_{12}\rangle
       +\frac1r\sum_{a=2}^r\E\langle fK_{1a}\rangle
\end{equation}
is at most
\begin{equation}\label{up:GGerror}
 2\rho+\frac{C(1+\sqrt{B_n})}{\rho s_n\tau_\ell}.
\end{equation}
This follows by using the trivial bound on $u<\rho$ and division by
$u$ on the complement. Choosing $\rho\downarrow0$ slowly makes the
bound vanish. All products of expectations in \eqref{up:GG} are taken
at the same fixed complete coefficient vector.

Powers of nonnegative rational linear forms span all nonconstant
polynomials in $(R,\widetilde C,Q)$, and the identity for a constant
test is automatic. Polynomial approximation therefore extends the
limiting identities to every continuous function of one vector overlap.
To state the limiting conclusion correctly, let $\nu_{n,\boldsymbol u}$
be the disorder-averaged law of the entire vector array at fixed
coefficients, and let $\mathscr L_n$ be the distribution of
$\nu_{n,\boldsymbol u}$ when $\boldsymbol u$ is sampled. The bounded
array space is compact, as is its space of probability laws. Along a
subsequence $\mathscr L_n\Rightarrow\mathscr L$.
The absolute defect of each continuous finite-array test is a continuous
function of the law $\nu$, including its nonlinear product of
expectations. By \eqref{up:GGerror}, its integral under $\mathscr L$ is
zero. A countable dense family of tests, followed by uniform approximation
and the monotone class theorem, proves that $\mathscr L$-almost every
component $\nu$ satisfies the full vector Ghirlanda--Guerra identities.
Every component is weakly exchangeable, has positive semidefinite
coordinate arrays, and has $Q$-marginal \eqref{up:Qlaw}, since these
properties are closed. This assertion holds along arbitrary deterministic
sequences of bounded controls and truncated times. It is an assertion
about the component laws, rather than about their average.

\subsection{Synchronization and response calibration}

\begin{lemma}\label{up:synchronization}
Let a bounded weakly exchangeable vector array have positive semidefinite
coordinate arrays and satisfy the full vector Ghirlanda--Guerra identities.
The support of its two-replica off-diagonal distribution is totally ordered
coordinatewise. Its coordinates are consequently deterministic
nondecreasing functions of their sum.
\end{lemma}

\begin{proof}
Every nonnegative rational combination of the coordinates is a bounded
positive semidefinite array satisfying the scalar identities.
The Dovbysh--Sudakov representation gives a random bounded Hilbert-space
measure representing its off-diagonals. The sphere-support theorem makes
the squared radius of this directing measure the deterministic top of the
overlap support. These are Theorems~1 and~6 of
\cite{Panchenko2014SK}; the identities in its equation~(86) test only
off-diagonals. Thus a variable original diagonal presents no obstruction.
After scaling into the unit ball, Theorem~1 of
\cite{Panchenko2013Ultrametricity} implies ultrametricity.

Suppose two support vectors $y,z$ cross: $y_i>z_i$ and $y_j<z_j$.
Choose disjoint neighborhoods $U,V$ preserving both strict inequalities.
Vector Ghirlanda--Guerra with two replicas gives
\[
 \Prob(X_{12}\in U,\ X_{13}\in V)
       =\tfrac12\Prob(X_{12}\in U)\Prob(X_{12}\in V)>0.
\]
On this event, coordinate ultrametricity forces
$X_{23,i}=X_{13,i}$ and $X_{23,j}=X_{12,j}$: the third overlap in an
ultrametric triple with two unequal overlaps equals their minimum.
It follows that $X_{23,i}+X_{23,j}$ is strictly smaller than both
$X_{12,i}+X_{12,j}$ and $X_{13,i}+X_{13,j}$. This contradicts
ultrametricity for the sum of coordinates $i,j$.
The support is a chain. Its coordinate sum is injective, and each
coordinate is a nondecreasing $1$-Lipschitz function of that sum on
the support; extend the functions between support points.
\end{proof}

Apply the lemma separately to each component of $\mathscr L$. Since
$Q$ takes the ordered values $q_i$ with positive probabilities $\pi_i$,
it gives
\begin{equation}\label{up:componentsign}
 \operatorname{Cov}_\nu(R_{12},C_{12}\mid Q_{12}=q_i)\geq0,
 \qquad
 0\leq\E_\nu[C_{12}\mid Q_{12}=q_0]\leq\cdots\leq
            \E_\nu[C_{12}\mid Q_{12}=q_k].
\end{equation}
For the covariance, condition on $Q=q_i$ and take two independent
copies of the scalar parameter of the chain. Twice the covariance is
the expectation of the product of the two coordinate differences,
which is nonnegative. The conditional means are ordered because every
support point at a lower $Q$ precedes every point at a higher $Q$.

We now identify controls for which this sign survives coefficient
averaging. Set $d_0=a_0$ and $d_i=a_i-a_{i-1}$ for $i\geq1$, and,
for $0<\eta\leq1$, define
\begin{equation}\label{up:envelope}
 M_{n,\eta}(t)=\max_{d_0,\ldots,d_k>0}
 \left\{\overline V_n(t,a)+\eta\sum_{i=0}^k\log d_i\right\}.
\end{equation}

\begin{lemma}\label{up:maximizers}
All maximizers in \eqref{up:envelope} satisfy $a_k\leq L$, with $L$
independent of $n,t,\eta\in(0,1]$. For each fixed $\eta>0$ they also
satisfy $d_i\geq\delta_\eta>0$, uniformly in $n,t$.
At every such maximizer, if
\[
 r_{n,i}(\boldsymbol u)=
   \frac{\E\langle R_{12}\ind\{Q_{12}=q_i\}\rangle_{\boldsymbol u}}{\pi_i},
 \qquad \overline r_{n,i}=\E_{\boldsymbol u}r_{n,i}(\boldsymbol u),
\]
then, for each fixed $\eta>0$,
\[
 \sup_{\boldsymbol u}\max_{0\leq i\leq k}
 |r_{n,i}(\boldsymbol u)-\overline r_{n,i}|
 \leq C_\eta n^{\gamma-1/2}
\]
for all sufficiently large $n$, uniformly in truncated time $t$ and in the
choice of maximizer. The constant may depend on the fixed profile,
temperature, density bound, time cutoff and $\eta$.
\end{lemma}

\begin{proof}
Let $S_n(a)$ denote \eqref{up:interpolation} with no row terms.
Writing the spin field as the sum of its independent tree increments,
using $\|\sigma\|=\sqrt n$, and applying the marking recursion gives
\begin{equation}\label{up:coercive}
 S_n(a)\leq\sum_{i=0}^k\sqrt{d_i}
                       +\frac12\sum_{i=1}^k x_i d_i.
\end{equation}
Indeed, $\E\|Z_n\|\leq\sqrt n$ and
$\log\E e^{u\|Z_n\|}\leq u\sqrt n+u^2/2$; the root term is
averaged directly, whereas each other increment is integrated with
mass $x_i$ and divided by $x_i$.
Moreover,
\[
 \mathcal B(q,a)-\sum_{i=1}^k x_i d_i
   =\sum_{i=0}^k d_i\sum_{j\geq i}\pi_j(1-q_j)
   \geq(1-x_k)(1-q_k)a_k,
\]
with $x_k=x_0=0$ when $k=0$.
The row potentials are nonpositive and \eqref{up:cost} is nonpositive,
so \eqref{up:coercive} implies
\[
 \overline V_n(t,a)
 \leq\sqrt{(k+1)a_k}-\tfrac12(1-x_k)(1-q_k)a_k.
\]
The logarithmic barrier preserves coercivity uniformly for
$0<\eta\leq1$. At $d_i=1$, the row bound $U_t\geq-\beta$,
the nonnegative spin pressure, and \eqref{up:cost} give a uniform
lower bound on the maximum. This proves the bound $a_k\leq L$.
For fixed $\eta$, the barrier tends to $-\infty$ as any increment
tends to zero, proving the positive lower bound on $d_i$ and existence
of an interior maximizer.

Write $b_i=\sqrt{d_i}$ and
$f_{n,\boldsymbol u}(b)=F_{n,\boldsymbol u}(t,a(b))$.
This is convex in $b$, since the spin field is linear in $b$ and
the score perturbation is independent of $b$. Equation~\eqref{up:cost}
gives
\[
 \sup_{b,\boldsymbol u}|f_{n,\boldsymbol u}(b)-\overline f_n(b)|
       \leq\epsilon_n,\qquad \epsilon_n=Cs_n^2/n\longrightarrow0.
\]
At a maximizer $b_n$, set
$G_\eta(b)=\mathcal B(q,a(b))/2-\eta\sum_i\log(b_i^2)$.
Its second derivatives are bounded on a fixed neighborhood of all
maximizers, for fixed $\eta$. Convex secants and
$\overline f_n(b_n+h e_i)-\overline f_n(b_n)
 \leq G_\eta(b_n+h e_i)-G_\eta(b_n)$ imply
\begin{equation}\label{up:responsegradient}
 |\partial_{b_i}f_{n,\boldsymbol u}(b_n)
                   -\partial_{b_i}G_\eta(b_n)|
       \leq C_\eta h+2\epsilon_n/h.
\end{equation}
Use the backward secant for the other inequality and choose
$h=\sqrt{\epsilon_n}$. The same derivative of $\overline f_n$ equals
$\partial_{b_i}G_\eta$ by stationarity.

Gaussian differentiation of the independent spin field, together with
\eqref{up:Qlaw}, gives exactly
\begin{equation}\label{up:spinderivative}
 \partial_{b_i}f_{n,\boldsymbol u}
   =b_i\left(1-\sum_{j\geq i}\pi_j r_{n,j}(\boldsymbol u)\right).
\end{equation}
Subtracting its coefficient average in \eqref{up:responsegradient} gives
\[
 \left|\sum_{j\geq i}\pi_j
       (r_{n,j}(\boldsymbol u)-\overline r_{n,j})\right|
 \leq C_\eta\sqrt{\epsilon_n}.
\]
Subtract adjacent tail sums, with the empty tail set to zero. Since
$b_i\geq\sqrt{\delta_\eta}$ and $\pi_i>0$, this proves the stated bound
with $\sqrt{\epsilon_n}=O(n^{\gamma-1/2})$. The barrier parameter remains
fixed throughout this dimension limit.
The maximizing control is deterministic: it maximizes the averaged
pressure before any coefficient is varied.
\end{proof}

At a maximizer, take any subsequential law $\mathscr L$ constructed
above, also taking a subsequence on which the finitely many
$\overline r_{n,i}$ converge. Lemma~\ref{up:maximizers} implies that
the limiting conditional spin means are the same, say $r_i$, in
$\mathscr L$-almost every component. Let $\overline\nu=\int\nu\,
\mathscr L(\dd\nu)$. Since every component has the same $Q$-marginal,
\begin{align*}
 \operatorname{Cov}_{\overline\nu}(R,C\mid Q=q_i)
 &=\int\operatorname{Cov}_\nu(R,C\mid Q=q_i)\,\mathscr L(\dd\nu)\\
 &\quad+\operatorname{Cov}_{\mathscr L}
       \big(r_i,\E_\nu[C\mid Q=q_i]\big)\geq0.
\end{align*}
The last covariance vanishes. The averaged conditional score means
are also nonnegative and nondecreasing by \eqref{up:componentsign}.
All these quantities are continuous functions of the array law:
the indicators of the finitely many $q_i$ can be replaced by
interpolation polynomials on their support. Compactness and contradiction
therefore give, uniformly at all maximizers and truncated times,
\begin{equation}\label{up:uniformsign}
 K_n:=\sum_i\pi_i\operatorname{Cov}_{\overline\nu_n}(R,C\mid Q=q_i)
       \geq-o_n(1),\qquad
 \operatorname{dist}(c_n,\mathcal C)\longrightarrow0,
\end{equation}
where $c_{n,i}=\E_{\overline\nu_n}[C\mid Q=q_i]$ and
$\mathcal C=\{0\leq c_0\leq\cdots\leq c_k\}$; here
$\operatorname{dist}(c,\mathcal C)=\inf_{v\in\mathcal C}\|c-v\|_2$
is the Euclidean distance to this cone.

\subsection{Differentiation of the envelope and its endpoints}

Lemma~\ref{ud:derivative} proves the local derivative formula
\begin{equation}\label{up:timederivative}
 \partial_t\overline F_n(t,a)
  =-\tfrac12\E_{\boldsymbol u}\E\langle(R_{12}-Q_{12})C_{12}\rangle
       +O_\varepsilon(s_n^2/n+s_n^4/n^2).
\end{equation}
There is no error in the spin derivative. In covariance coordinates it is
\begin{equation}\label{up:agradient}
 \partial_{a_i}\overline V_n=-\tfrac12\pi_i
                  (\overline r_{n,i}-q_i).
\end{equation}
Decomposing the expectation in \eqref{up:timederivative} conditionally
on $Q$ gives
\begin{equation}\label{up:transport}
 \partial_t\overline V_n
   =-\tfrac12K_n+\nabla_a\overline V_n\cdot c_n+o_n(1).
\end{equation}
At a maximizer, stationarity says
$\nabla_a\overline V_n=-\eta\nabla_a\sum_i\log d_i$.
For $c\in\mathcal C$,
\[
 \nabla_a\left(\sum_i\log d_i\right)\cdot c
       =\frac{c_0}{d_0}+\sum_{i=1}^k\frac{c_i-c_{i-1}}{d_i}\geq0.
\]
Since \eqref{up:agradient} uniformly bounds the gradient, the distance
error in \eqref{up:uniformsign} is harmless. Thus
\begin{equation}\label{up:maxderivative}
 \partial_t\overline V_n(t,a)\leq o_n(1)
\end{equation}
uniformly over all maximizers.

We justify using this estimate at a time-dependent maximizing control.
Fix $n$ and a compact set $\mathcal A$ in the interior of the control
cone. We first retain only finitely many feature kernels and couple
all parameters in $[0,1-\varepsilon]\times\mathcal A$ by the same
standard Gaussian decorations. On each leaf, the resulting exponent
$H_\gamma(t,a,\sigma)$ is continuous on this compact parameter set
times $\mathbb S_n$. There is a random bound
\[
 \sup_{t,a,\sigma}e^{H_\gamma(t,a,\sigma)}\leq E_\gamma,
 \qquad \sup_\gamma\E E_\gamma<\infty.
\]
Indeed, the row exponent and compensator are nonpositive. The spin
exponent is bounded by a fixed constant times the sum of the norms
of the finitely many standard Gaussian spin increments on the leaf's
path. Each retained feature field is a finite Gaussian tensor
polynomial on spin and score vectors of bounded norm, so its supremum
is bounded by a fixed constant times the sum of the absolute values
of its Gaussian coefficients on that path. These finite Gaussian
vectors have exponential moments of every fixed linear norm; their
laws and the constants are the same for every leaf. The constants
here may depend on $n$, the feature cutoff and $\mathcal A$.

The weights depend only on the Poisson points and are independent of
these decorations. Consequently
$\E\sum_\gamma w_\gamma E_\gamma\leq\sup_\gamma\E E_\gamma<\infty$.
This sum dominates all leaf integrals, uniformly in the parameters.
Dominated convergence on the sphere and for the leaf sum gives a
continuous partition. A single leaf of positive weight supplies strict
positivity simultaneously at all parameters. Hence the partition has
a positive minimum on the compact parameter set almost surely. The same argument for replica numerators, dominated by
$(\sum_\gamma w_\gamma E_\gamma)^r$, proves samplewise continuity of
bounded continuous $r$-replica insertions. Integrability of this last
power is unnecessary: the normalized insertion is bounded, so its
expectation is continuous by bounded convergence. Expected logarithms
are continuous by the uniform logarithmic integrability in
Lemma~\ref{ud:limits}.

In the derivative formula of Lemma~\ref{ud:derivative}, average the
auxiliary Gaussian directions first. Their second moments replace
each product $Y^aY^b$ by $R_{ab}$ or $Q_{ab}$. Thus the remaining
insertions are bounded continuous functions of finitely many states
and the parameters; no derivative of $\sqrt t$ occurs in these
insertions. The spin derivative has this property as well.
Lemma~\ref{ud:limits} bounds the change in any such averaged insertion
when the omitted feature covariance is restored, uniformly over the
compact parameter set. Its kernel derivative tails vanish uniformly
by $\sum_\ell\tau_\ell^2(1+p_\ell)^2<\infty$.
Hence the full right sides of the time and control derivative formulas
are uniform limits of continuous functions. Their integrated identities
identify them with the actual continuous derivatives in the interior,
with one-sided time derivatives at the interval endpoints. All bounds
also hold uniformly in the coefficient vector, so coefficient averaging
preserves these conclusions. In particular the formulas hold at every
interior maximizing control.

For fixed $\eta$, all maximizers lie in a common compact set and the
time derivatives are uniformly bounded there. The maximum in
\eqref{up:envelope} is consequently Lipschitz. At its differentiability
times, its derivative is bounded by the largest time derivative among
its maximizers. This follows directly by taking difference quotients
and a convergent subsequence of maximizing controls. Integration of
\eqref{up:maxderivative} yields
\begin{equation}\label{up:envelopecomparison}
 M_{n,\eta}(1-\varepsilon)\leq M_{n,\eta}(0)+o_n(1).
\end{equation}

At $t=0$, independence of the row and spin coordinates factorizes every
step of the marking recursion, and hence
\[
 F_n(0,a)=\alpha_n\mathcal J_\beta(q,x)+S_n(a).
\]
The spherical inequality of Lemma~\ref{ue:spherical}, uniformly for
$a_k\leq L$, gives
\begin{equation}\label{up:zeroendpoint}
 M_{n,\eta}(0)
 \leq\mathcal P_{\alpha_n,\beta}(q,x)
           +\eta(k+1)\log\max(1,L)+o_n(1),
\end{equation}
where we also used the upper bound in \eqref{up:cost}.

For the other endpoint choose $d_i=\eta$. Adding an independent spin
field along the covariance path $ra$, $0\leq r\leq1$, has normalized
pressure derivative
\[
 \tfrac12\big(a_k-\E\langle R_{12}a(Q_{12})\rangle\big)\geq0,
\]
because $R_{12}a(Q_{12})\leq a_k$ pointwise. Thus
$F_n(t,a)\geq F_n(t,0)$, while $\mathcal B(q,a)\leq a_k=(k+1)\eta$.
Equation~\eqref{up:cost} implies
\begin{equation}\label{up:smallcontrol}
 M_{n,\eta}(t)\geq F_n(t,0)-\tfrac12(k+1)\eta
                +\eta(k+1)\log\eta-Cs_n^2/n.
\end{equation}

Finally, the discontinuous endpoint can be removed uniformly in $a$.
Replace each row factor $e^{U_t(X_i)}$ by its $t=1$ factor, one row at
a time, and delete that row to define the comparison Gibbs measure.
Both factors lie in $[e^{-\beta},1]$, so their logarithmic partition
increments differ by at most $e^\beta$ times the Gibbs mean of their
absolute difference. The deleted-row measure is independent of $W_i$
and its entire $z_i$ tree. After including the residual Gaussian in
\eqref{up:channel}, the two row arguments, at each fixed state, are
standard normals with correlation $\sqrt t$. Their signs disagree with
probability $\pi^{-1}\arccos\sqrt t$. Summing the replacements proves
\begin{equation}\label{up:stependpoint}
 |F_n(t,a)-F_n(1,a)|
 \leq\alpha_n(e^\beta-1)\frac{\arccos\sqrt t}{\pi}.
\end{equation}
Combining \eqref{up:envelopecomparison}, \eqref{up:zeroendpoint},
\eqref{up:smallcontrol}, and \eqref{up:stependpoint}, and then taking
$n\to\infty$, $\eta\downarrow0$, and $\varepsilon\downarrow0$ in
that order proves Theorem~\ref{up:pressure}. The profile and temperature
are fixed throughout these limits.

\section{The spherical endpoint inequality}\label{ue:section}

Keep the masses $x_i$ fixed and allow
$0\leq q_0\leq\cdots\leq q_k<1$ in this appendix.
For an ordered nonnegative covariance path $a$, let
\[
 S_n(a)=\frac1n\E\log\sum_\gamma w_\gamma
       \int_{\mathbb S_n}e^{g(\gamma)\cdot\sigma}\,\mu_n(\dd\sigma),
 \qquad
 \mathcal B(q,a)=a_k-\sum_{i=0}^k\pi_iq_i a_i,
\]
where the independent coordinates of $g$ have the tree covariance $a$.

\begin{lemma}\label{ue:spherical}
For every $L<\infty$, uniformly over $0\leq a_0\leq\cdots\leq a_k\leq L$,
\[
 S_n(a)-\tfrac12\mathcal B(q,a)
       \leq\mathcal E(q,x)+O_L\!\left(\frac{\log(n+1)}n\right).
\]
\end{lemma}

\begin{proof}
We first establish the required convexity of $\mathcal E$.
Write $\lambda_i=1-q_k+\sum_{j>i}x_j(q_j-q_{j-1})$ and define
the canonical covariance path by
\begin{equation}\label{ue:canonical}
 \widehat q_0=\frac{q_0}{\lambda_0^2},\qquad
 \widehat q_i-\widehat q_{i-1}
       =\frac{q_i-q_{i-1}}{\lambda_{i-1}\lambda_i}.
\end{equation}
Direct differentiation gives
\begin{equation}\label{ue:gradient}
 \partial_{q_i}\mathcal E(q,x)=\tfrac12\pi_i\widehat q_i.
\end{equation}
Here is also an explicit Hessian calculation. The affine form of
$\lambda_i$ is
$\lambda_i=1-x_{i+1}q_i-\sum_{j>i}\pi_jq_j$.
For a direction $z$, put
$L_i=-x_{i+1}z_i-\sum_{j>i}\pi_jz_j$ and, for $i\geq1$,
$c_i=x_i^{-1}-x_{i+1}^{-1}>0$.
For $k\geq1$, telescoping the defining logarithms gives
\[
 2\mathcal E=\frac{q_0}{\lambda_0}
              +\frac1{x_1}\log\lambda_0
              -\sum_{i=1}^k c_i\log\lambda_i.
\]
Its first derivatives telescope using
\[
 \frac1{x_i}\left(\frac1{\lambda_i}-\frac1{\lambda_{i-1}}\right)
       =\frac{q_i-q_{i-1}}{\lambda_{i-1}\lambda_i},
\]
which proves \eqref{ue:gradient}. Write $D^2\mathcal E[z,z]$ for the
second directional derivative in the overlap direction $z$, with the
masses $x$ fixed. Then
\[
 2D^2\mathcal E[z,z]
 =-\frac{2z_0L_0}{\lambda_0^2}
  +\left(\frac{2q_0}{\lambda_0^3}
                      -\frac1{x_1\lambda_0^2}\right)L_0^2
  +\sum_{i=1}^k\frac{c_iL_i^2}{\lambda_i^2}.
\]
Expanding the $L_i$ and cancelling cross terms gives
$-2z_0L_0-L_0^2/x_1+\sum_{i=1}^kc_iL_i^2=\sum_i\pi_i z_i^2$.
Consequently
\begin{equation}\label{ue:hessian}
 2D^2\mathcal E[z,z]
 =\lambda_0^{-2}\sum_i\pi_i z_i^2
     +2q_0\lambda_0^{-3}L_0^2
     +\sum_{i=1}^kc_iL_i^2(\lambda_i^{-2}-\lambda_0^{-2})
 \geq\sum_i\pi_i z_i^2.
\end{equation}
For $k=0$ the same formula, with empty sums and $L_0=-z_0$, follows
directly from $\mathcal E(q_0)=\frac12\log(1-q_0)+q_0/[2(1-q_0)]$.
All formulas extend to repeated overlaps and $q_0=0$ because the
$\lambda_i$ remain positive. Integrating the Hessian bound along a
segment in the ordered overlap domain yields
\begin{equation}\label{ue:tangent}
 \mathcal E(q,x)\geq\mathcal E(r,x)
       +\nabla\mathcal E(r,x)\cdot(q-r)
       +\tfrac14\sum_i\pi_i(q_i-r_i)^2.
\end{equation}

Every nonnegative ordered covariance path is canonical for a unique
overlap path. To see this, put $\Delta a_i=a_i-a_{i-1}$ and
$A_0=\sum_{i=1}^kx_i\Delta a_i$. For $b>A_0$, define
\begin{equation}\label{ue:inverse}
 b_i=b-\sum_{j>i}x_j\Delta a_j,\qquad
 r_0(b)=a_0/b_0^2,\qquad
 r_i(b)-r_{i-1}(b)=\frac{\Delta a_i}{b_{i-1}b_i}.
\end{equation}
There is a unique solution of
\begin{equation}\label{ue:normalization}
 r_k(b)+1/b=1.
\end{equation}
The left side is continuous and strictly decreasing to zero. It
diverges at its lower endpoint: if $a_0>0$ the first term diverges;
otherwise the first positive increment has $b_{i-1}\downarrow0$ and
$b_i\to x_i\Delta a_i>0$. If the entire path is zero, the equation
is simply $1/b=1$. The solution has $b\geq1$ and
$0\leq r_0\leq\cdots\leq r_k=1-1/b<1$.
The identity
$x_i\Delta r_i=1/b_{i-1}-1/b_i$ telescopes to
\begin{equation}\label{ue:inverse-lambda}
 \lambda_i(r)=1/b_i.
\end{equation}
Substitution in \eqref{ue:canonical} gives $\widehat r=a$.
Conversely a canonical path gives \eqref{ue:inverse}--\eqref{ue:normalization}
with $b=(1-r_k)^{-1}$, proving uniqueness.
If $a_k\leq L$, evaluation of the left side of
\eqref{ue:normalization} at $b=2L+2$ bounds it by
$L/(L+2)^2+1/(2L+2)<1$, and hence
\begin{equation}\label{ue:bbound}
 1\leq b\leq2L+2.
\end{equation}

For this $b$, replace the spherical integral at a leaf by
\[
 I_\gamma=\int_{\mathbb R^n}
       e^{-b\|\sigma\|^2/2+g(\gamma)\cdot\sigma}
                 \frac{\dd\sigma}{(2\pi)^{n/2}}
       =b^{-n/2}e^{\|g(\gamma)\|^2/(2b)}.
\]
At level $i$, Gaussian quadratic integration changes precision $b_i$
to $b_{i-1}=b_i-x_i\Delta a_i>0$. The marking recursion gives exactly
\begin{equation}\label{ue:Gaussian}
 \frac1n\E\log\sum_\gamma w_\gamma I_\gamma
   =-\frac12\log b+\frac12\sum_{i=1}^k\frac1{x_i}
                    \log\frac{b_i}{b_{i-1}}+\frac{a_0}{2b_0}
   =\mathcal E(r,x).
\end{equation}
Only conditional Gaussian moments at the successive levels are used;
the final root operation is expectation of a quadratic logarithm.

For each fixed vector $g$, the normalized spherical integral of
$e^{g\cdot\sigma}$ is nondecreasing in the sphere radius, by antipodal
symmetry and monotonicity of $\cosh$. Restrict the radial integral for
$I_\gamma$ to radii between $\sqrt n$ and $(1+h)\sqrt n$, with
$h=[n(1+b)]^{-1}$. This gives the leafwise bound
\[
 I_\gamma\geq c_{n,b}
       \int_{\mathbb S_n}e^{g(\gamma)\cdot\sigma}\,\mu_n(\dd\sigma),
 \qquad
 c_{n,b}=\frac{2^{1-n/2}n^{n/2}}{\Gamma(n/2)}
       \int_1^{1+h}u^{n-1}e^{-bnu^2/2}\,\dd u.
\]
The positive cascade weights preserve this inequality. Put
$\mathfrak e_n(a)=-n^{-1}\log c_{n,b}-(b-1)/2$.
The elementary Stirling bounds give
\[
 \frac{\frac12\log n+\frac12\log\pi+\log(1+b)-1}{n}
 \leq \mathfrak e_n(a)\leq
 \frac{\frac12\log n+\frac12\log\pi+\log(1+b)+1+1/n}{n}.
\]
For the upper bound use
$2^{1-n/2}n^{n/2}/\Gamma(n/2)\geq
 \sqrt{n/\pi}e^{n/2-1/(6n)}$, $u^{n-1}\geq1$,
$bnh\leq1$, and $bnh^2/2\leq1/(8n)$.
For the lower bound use the reverse prefactor bound and
$u^{n-1}\leq e$ on the shell. In view of \eqref{ue:bbound},
$|\mathfrak e_n(a)|=O_L(\log(n+1)/n)$ uniformly.

Finally,
\[
 \mathcal B(r,a)
  =(1-r_k)a_k+\sum_{i=1}^kx_i(r_i a_i-r_{i-1}a_{i-1})=b-1.
\]
For the last equality, the sequence $b_i r_i-a_i/b_i$ starts at zero
and its $i$th increment is $x_i(r_i a_i-r_{i-1}a_{i-1})$; then use
$1-r_k=1/b$. The shell comparison and \eqref{ue:Gaussian} therefore give
\[
 S_n(a)\leq\mathcal E(r,x)+\tfrac12\mathcal B(r,a)+\mathfrak e_n(a).
\]
Subtracting $\mathcal B(q,a)/2$ adds
$\frac12\sum_i\pi_i a_i(q_i-r_i)=\nabla\mathcal E(r,x)\cdot(q-r)$.
Equation~\eqref{ue:tangent} proves, in fact, the stronger bound
\[
 S_n(a)-\tfrac12\mathcal B(q,a)
 \leq\mathcal E(q,x)-\tfrac14\sum_i\pi_i(q_i-r_i)^2+\mathfrak e_n(a).
\]
\end{proof}

\section{Differentiating the score perturbation}\label{ud:section}

This appendix proves the derivative estimate used in
\eqref{up:timederivative}. The definitions of $U_t,\chi_t,C,K_\ell,D$
are those of \eqref{up:channel}--\eqref{up:kernels}, with a fixed
coefficient vector. We suppress that vector in the notation. The
amplitude $s\geq0$ is arbitrary in the calculation.

\begin{lemma}\label{ud:derivative}
Let $F_{n,s}(t,a)$ be the expected normalized log partition with
perturbation $sG_D-s^2D_{11}/2$. For $m\leq An$ and
$0\leq t\leq1-\varepsilon$, it is absolutely continuous in $t$ and
\begin{equation}\label{ud:statement}
 \partial_t F_{n,s}(t,a)
   =-\tfrac12\E\langle(R_{12}-Q_{12})C_{12}\rangle
                   +\mathcal R_{n,s}(t,a),\qquad
 |\mathcal R_{n,s}|\leq C\left(\frac{s^2}{n}+\frac{s^4}{n^2}\right).
\end{equation}
The identity holds in integrated form on the whole interval. The
constant is independent of $n,m,a$ and the coefficients within the stated
bounds. The independent spin-covariance derivatives have no error:
\begin{equation}\label{ud:spin}
 \partial_{a_i}F_{n,s}
 =\tfrac12\left(\ind\{i=k\}
            -\E\langle R_{12}\ind\{Q_{12}=q_i\}\rangle\right)
\end{equation}
in the interior of the covariance cone.
\end{lemma}

\begin{proof}
We first give a finite-state calculation. With all other disorder fixed,
let
\[
 \Psi(\theta)=\E_g\log\int
       e^{H_\theta+s g_{D_\theta}-s^2D_{\theta,11}/2}\,\dd\mu.
\]
For an $r$-replica observable define the insertion operators
\begin{align*}
 L_r(h)&=\sum_{a=1}^r h_a-rh_{r+1},\\
 T_r(K)&=\sum_{\substack{a,b\leq r\\a\ne b}}K_{ab}
       -2r\sum_{a=1}^rK_{a,r+1}+r(r+1)K_{r+1,r+2}.
\end{align*}
Gaussian covariance differentiation gives
\begin{align}
 \dot\Psi&=\E_g\langle\dot H_1\rangle
                      -\tfrac{s^2}{2}\E_g\langle\dot D_{12}\rangle,
                      \label{ud:first}\\
 \partial_\theta\E_g\langle f\rangle
   &=\E_g\left\langle\dot f+fL_r(\dot H)
                              +\tfrac{s^2}{2}fT_r(\dot D)\right\rangle.
                      \label{ud:observable}
\end{align}
For \eqref{ud:observable}, the usual covariance insertion before
compensation is
\[
 \sum_{a,b\leq r}\dot D_{ab}
 -2r\sum_{a\leq r}\dot D_{a,r+1}
 +r(r+1)\dot D_{r+1,r+2}-r\dot D_{r+1,r+1}.
\]
Differentiating the compensator subtracts
$\sum_{a\leq r}\dot D_{aa}-r\dot D_{r+1,r+1}$ and cancels every
diagonal term. In particular
$T_1(K)=-2K_{12}+2K_{23}$ and
$T_2(K)=2K_{12}-4K_{13}-4K_{23}+6K_{34}$.
Applying \eqref{ud:observable} to both terms of \eqref{ud:first} yields
\begin{align}
 \ddot\Psi
 &=\E_g\langle\ddot H_1+\dot H_1(\dot H_1-\dot H_2)\rangle\nonumber\\
 &\quad+\tfrac{s^2}{2}\E_g\langle
      \dot H_1T_1(\dot D)-\ddot D_{12}-\dot D_{12}L_2(\dot H)\rangle
      -\tfrac{s^4}{4}\E_g\langle\dot D_{12}T_2(\dot D)\rangle.
      \label{ud:second}
\end{align}
Only $D_\theta$, rather than its derivatives, must be positive semidefinite.
Singular finite covariances follow by adding a fixed positive diagonal
regularization and taking its size to zero.

The row covariance derivative is $R-Q$, which need not be positive
semidefinite. Introduce independent genuine Gaussian directions $Y_R$
and $Y_Q$ with covariances $R$ and $Q$, respectively, and write
\[
 \mathscr D f=\E_{Y_R}f(Y_R)-\E_{Y_Q}f(Y_Q).
\]
These directions are independent of the Gibbs disorder and occur only
in insertions. Thus $\mathscr D(Y^aY^b)=R_{ab}-Q_{ab}$ and
$\mathscr D((Y^a)^2)=v$.
For the row variation $X_i\mapsto X_i+\theta Y$, let
\[
 h_{i,Y}=\chi_t(X_i)Y,\qquad
 h''_{i,Y}=\chi_t'(X_i)Y^2,
 \qquad D_{i,Y}=\partial_\theta D,\quad
 D_{i,YY}=\partial_\theta^2D.
\]
Also let $\dot D$ denote the explicit time derivative of $D$ at fixed
row arguments. The heat equations are
\begin{equation}\label{ud:heat}
 \partial_tU_t=-\tfrac v2(\chi_t'+\chi_t^2),\qquad
 \partial_t\chi_t=-\tfrac v2(\chi_t''+2\chi_t\chi_t').
\end{equation}
Apply \eqref{ud:second} to each row and contract its second derivative
with $\mathscr D/2$. The terms independent of $s$ are
\[
 \frac12\sum_{i=1}^m\E\left\langle
       v(\chi_t'+\chi_t^2)(X_i^1)
       -(R_{12}-Q_{12})\chi_t(X_i^1)\chi_t(X_i^2)
                             \right\rangle.
\]
The first term cancels the explicit derivative of $U_t$ by
\eqref{ud:heat}. After division by $n$, the second is the principal
term in \eqref{ud:statement}.

The remaining terms can be recorded exactly. Set, with $h=h_{i,Y}$,
\[
 A_i(Y)=h_1T_1(D_{i,Y})-(D_{i,YY})_{12}
                         -(D_{i,Y})_{12}L_2(h),\qquad
 B_i(Y)=(D_{i,Y})_{12}T_2(D_{i,Y}).
\]
Then
\begin{equation}\label{ud:remainder}
 \mathcal R_{n,s}
 =-\frac{s^2}{2n}\E\langle\dot D_{12}\rangle
       +\frac{s^2}{4n}\sum_{i=1}^m\mathscr D\E\langle A_i(Y)\rangle
       -\frac{s^4}{8n}\sum_{i=1}^m\mathscr D\E\langle B_i(Y)\rangle.
\end{equation}
The $s^4$ term comes from differentiating the two-replica kernel
insertion. All Gibbs brackets here are those at the unvaried row
realization; the auxiliary Gaussian expectation does not change them.

We bound each insertion directly. With
$\chi_a=\chi_t(X_i^a)$ and spatial primes, the one-row score derivatives
are
\begin{align*}
 (\widetilde C_{i,Y})_{12}
 &=\frac{\chi'_1\chi_2Y^1+\chi_1\chi'_2Y^2}{nB_\varepsilon},\\
 (\widetilde C_{i,YY})_{12}
 &=\frac{\chi''_1\chi_2(Y^1)^2
              +2\chi'_1\chi'_2Y^1Y^2
              +\chi_1\chi''_2(Y^2)^2}{nB_\varepsilon}.
\end{align*}
Writing $k_\ell=\vartheta_{\ell,R}R+
\vartheta_{\ell,C}\widetilde C+\vartheta_{\ell,Q}Q$, the chain rule gives
\begin{align*}
 D_{i,Y}&=\sum_\ell\tau_\ell^2u_\ell^2p_\ell k_\ell^{p_\ell-1}
                         \vartheta_{\ell,C}\widetilde C_{i,Y},\\
 D_{i,YY}&=\sum_\ell\tau_\ell^2u_\ell^2
    \left[p_\ell k_\ell^{p_\ell-1}\vartheta_{\ell,C}\widetilde C_{i,YY}
       +p_\ell(p_\ell-1)k_\ell^{p_\ell-2}
                      \vartheta_{\ell,C}^2\widetilde C_{i,Y}^2\right].
\end{align*}
The second term is zero for $p_\ell=1$, without evaluating a negative
power. Since $|k_{\ell,ab}|\leq1$, summability and the bounded derivatives
of $\chi_t$ give
\begin{equation}\label{ud:kernelbounds}
 |(D_{i,Y})_{ab}|\leq\frac Cn(|Y^a|+|Y^b|),\quad
 |(D_{i,YY})_{ab}|\leq\frac Cn((Y^a)^2+(Y^b)^2),\quad
 |\dot D_{ab}|\leq C.
\end{equation}
For the last bound, explicitly differentiate
\[
 \dot{\widetilde C}_{ab}
   =\frac1{nB_\varepsilon}\sum_{i=1}^m
       \left[(\partial_t\chi_t)(X_i^a)\chi_t(X_i^b)
                      +\chi_t(X_i^a)(\partial_t\chi_t)(X_i^b)\right]
\]
and use $m\leq An$ and \eqref{ud:heat}.
For either direction, at fixed replicas,
$\E_Y(Y^a)^2\leq1$ and $\E_Y|Y^aY^b|\leq1$.
Equations~\eqref{ud:kernelbounds} and the displayed insertion operators
therefore imply
\[
 |\mathscr D\E\langle A_i(Y)\rangle|\leq C/n,
 \qquad
 |\mathscr D\E\langle B_i(Y)\rangle|\leq C/n^2.
\]
Substitution in \eqref{ud:remainder} proves \eqref{ud:statement}.

To remove the finite-state restrictions, apply Lemma~\ref{ud:limits}
below. In particular, first average the auxiliary directions in
\eqref{ud:remainder}. Every resulting insertion is a bounded continuous
function of finitely many states at finite feature cutoff, and the
bounds just proved hold uniformly under spherical quadrature and leaf
truncation. The feature tails of the kernels and of their derivatives
vanish uniformly. The lemma therefore passes the integrated
finite-state identity to the full model, proving absolute continuity
and \eqref{ud:statement}.

Finally, the feature kernel is independent of the independent spin-field
coefficients and their covariance path. Its differentiation introduces
no extra term in a spin covariance variation. The spin covariance
derivative is $nR_{12}\ind\{Q_{12}=q_i\}$, with self derivative
$n\ind\{i=k\}$. Ordinary Gaussian differentiation and division by $n$
give \eqref{ud:spin}. Admissible one-sided derivatives at the boundary
follow by continuity.
\end{proof}

\begin{lemma}[State-space approximation]\label{ud:limits}
Fix $n,m,s$, a coefficient vector $\boldsymbol u\in[0,1]^{\mathbb N}$,
$\varepsilon>0$, and $L<\infty$. For $0\leq t\leq1-\varepsilon$ and
$0\leq a_0\leq\cdots\leq a_k\leq L$, the compensated partition with
covariance $D$ is finite and strictly positive almost surely.
Its expected logarithm, and the expected bounded replica insertions
appearing in \eqref{ud:first}--\eqref{ud:remainder} and
\eqref{ud:spin}, after averaging auxiliary Gaussian directions, are obtained by the following successive limits:
spherical probability quadratures, finite leaf sets, and finite feature
sums. The quadratures may be chosen antipodally symmetric (invariant under
$\sigma\mapsto-\sigma$), and finite leaf weights are normalized to have sum one. The integrated derivative
identities pass through these limits.

More precisely, put $D^{(N)}=\sum_{\ell\leq N}\tau_\ell^2u_\ell^2K_\ell$,
$T=\sum_\ell\tau_\ell^2$, and
$\delta_N=\sum_{\ell>N}\tau_\ell^2$.
The unnormalized expected log partitions for $D$ and $D^{(N)}$
differ by at most $s^2\delta_N/2$. If $f$ is a bounded $r$-replica
insertion, continuous in the spherical coordinates for each fixed
leaf tuple and independent of the feature Gaussian coefficients, then
\begin{equation}\label{ud:featuretail}
 \left|\E\langle f\rangle_D-\E\langle f\rangle_{D^{(N)}}\right|
 \leq C_r s^2\|f\|_\infty\delta_N.
\end{equation}
These estimates are uniform in the parameters just specified and in
the leaf and sphere approximations. Logarithms of all the approximating
partitions are uniformly integrable for these fixed $n,m,s,L$.
The same assertions hold when one coefficient is in $[-1,2]$, with
$T$ and $\delta_N$ replaced by $4T$ and $4\delta_N$ in the bounds.
\end{lemma}

\begin{proof}
For clarity, all Gaussian tensor coefficients defining the feature
fields are realized independently of the weights, rows, and spin
field; their state arguments may depend on the rows. Enumerate the countable cascade leaves and take increasing
finite sets exhausting them, using an enumeration depending only on the
Poisson points. Their total original weight tends to one almost surely.
Take deterministic, finitely supported antipodally symmetric probability
measures on $\mathbb S_n$ converging weakly to $\mu_n$.
At fixed feature cutoff and finitely many leaves, the integrands are
continuous on the sphere, so these quadratures converge samplewise.
After this limit the leaf sums converge by positivity, and their
normalizing total weights converge to one. The limits are finite
almost surely: Gaussian compensation, conditional on the other
disorder, makes the expected perturbed partition equal to the base
partition, whose integrability is verified below. For bounded replica
insertions continuous in the spherical coordinates, the numerators
converge by the same argument, using absolute domination by a constant
times the product partition. The denominators converge to strictly
positive limits at every finite feature cutoff. Thus their normalized
values converge samplewise, and their expectations converge by bounded
convergence.

We give bounds that apply to every approximation. Let $Z_0$ denote
its base partition and $Z_{\rm sp}$ the same partition with only the
spin field in its exponent. Antipodal symmetry and
$-\beta\leq U_t\leq0$ give
\begin{equation}\label{ud:baselogbounds}
 e^{-\beta m}\leq Z_0\leq Z_{\rm sp},\qquad
 \E Z_{\rm sp}=e^{na_k/2}\leq e^{nL/2}.
\end{equation}
For the last identity, condition on the weights and average the spin
Gaussians at a fixed state, whose variance is $na_k$.
It follows that $\log Z_0\geq-\beta m$ and
$\Prob(\log Z_0>r)\leq e^{nL/2-r}$.
In particular these logarithms have uniformly bounded moments of every
fixed order.

Condition now on the weights, rows and spin field, and let $\nu_0$
be the normalized base Gibbs measure of the approximation in question.
For a finite feature cutoff write
\[
 Y_N=\frac{Z_{s,N}}{Z_0}
   =\int e^{sG^{(N)}(\xi)-s^2D^{(N)}(\xi,\xi)/2}\,\nu_0(\dd\xi),
 \qquad \xi=(\sigma,\gamma).
\]
Let $\E_g$ denote expectation over the feature Gaussians only. Since
$|D^{(N)}(\xi,\xi')|\leq T$, compensation and Jensen's inequality give
\begin{equation}\label{ud:partitionmoments}
 \E_gY_N=1,\qquad
 \E_gY_N^2=\iint e^{s^2D^{(N)}(\xi,\xi')}
                  \,\nu_0(\dd\xi)\nu_0(\dd\xi')\leq e^{s^2T}.
\end{equation}
\begin{equation}\label{ud:jensenbound}
 \log Y_N\geq s\int G^{(N)}\,\dd\nu_0-s^2T/2,\qquad
 \Var_g\left(s\int G^{(N)}\,\dd\nu_0\right)\leq s^2T.
\end{equation}
Consequently, for $r>0$,
\[
 \Prob_g(\log Y_N>r)\leq e^{-r},\qquad
 \Prob_g(\log Y_N<-s^2T/2-r)
       \leq e^{-r^2/(2s^2T)}
\]
when $s^2T>0$; at $s=0$ one has $Y_N=1$.
Together with \eqref{ud:baselogbounds}, these bounds prove uniform
integrability, indeed uniform bounds for every fixed moment, of
$\log Z_{s,N}$. They are uniform in time, bounded controls, and both
state-space approximations. Hence the samplewise quadrature and leaf
limits above also pass through expected logarithms.

It remains to construct and identify the infinite-feature limit.
We may now use the full sphere and all leaves in $\nu_0$.
Conditional on this measure and the rows, realize
$G^{(N)}=\sum_{\ell\leq N}\tau_\ell u_\ell g_\ell$ with independent
fields of covariance $K_\ell$.
Then $(Y_N)$ is a martingale for the increasing sigma-fields generated
by these fields, because each new exponential factor has conditional
mean one at each state. Equation~\eqref{ud:partitionmoments} makes it
bounded in $L^2$, so $Y_N\to Y$ almost surely and in $L^2$, with
$\E_gY=1$. This limit cannot vanish: the Gaussian variables
$s\int G^{(N)}\,\dd\nu_0$ have independent increments and total
variance at most $s^2T$, and therefore converge almost surely to a
finite real variable. Their Jensen bounds \eqref{ud:jensenbound}
then imply $Y>0$ almost surely.

For identification of the Gibbs measure, set
$W_N(\xi)=\exp\{sG^{(N)}(\xi)-s^2D^{(N)}(\xi,\xi)/2\}$.
For $M>N$, independence of the omitted fields gives
\[
 \E_g\int|W_M-W_N|^2\,\dd\nu_0
 =\int e^{s^2D^{(N)}(\xi,\xi)}
       \left(e^{s^2(D^{(M)}-D^{(N)})(\xi,\xi)}-1\right)
                   \,\nu_0(\dd\xi)
 \leq e^{s^2T}(e^{s^2\delta_N}-1).
\]
Thus $W_N$ converges in $L^2(\Prob_g\otimes\nu_0)$ to a density $W$.
At $\nu_0$-almost every state the Gaussian series has summable
variances and converges almost surely; consequently $W$ is precisely
the strictly positive compensated exponential of the full field.
Moreover $Y=\int W\,\dd\nu_0$, by $L^1$ convergence. Since $Y>0$,
the normalized measures $W_N\nu_0/Y_N$ converge in total variation
in probability to $W\nu_0/Y$. Bounded replica averages therefore
converge in probability and in expectation. The logarithmic tail
bounds above identify $\E\log Z_s$ with the limit of
$\E\log Z_{s,N}$.

For quantitative control, interpolate between $D^{(N)}$ and
$D^{(M)}$ by a convex combination. Their difference has entrywise
absolute value at most $\delta_N$. The elementary finite-state
identities \eqref{ud:first} and \eqref{ud:observable}, with the base
exponent and insertion fixed, bound the change of expected logarithm
by $s^2\delta_N/2$ and that of an $r$-replica insertion by
$C_rs^2\|f\|_\infty\delta_N$. The quadrature and leaf limits already
proved extend these estimates to the full state space. Letting
$M\to\infty$ gives the asserted bounds, including
\eqref{ud:featuretail}. If the insertion itself depends on the cutoff,
one adds the supremum norm of its difference from its limit.

Finally, the tails of the kernel derivatives in
\eqref{ud:kernelbounds}, after averaging the auxiliary directions,
are bounded by constants times
$\sum_{\ell>N}\tau_\ell^2(1+p_\ell)^2$, which tends to zero.
Products of these kernels have the same convergence because each
factor is uniformly bounded. Therefore every term of the derivative
identity converges pointwise in time under the successive limits.
Its absolute value is bounded on the truncated time interval by the
finite-state bounds in \eqref{ud:statement} and
\eqref{ud:kernelbounds}, uniformly in all approximations, for fixed
$n,m,s$. Dominated convergence passes each integrated identity to
the limit on every subinterval. The same reasoning applies to spin
covariance variations on compact control intervals. This establishes
the asserted integrated identities without requiring uniform
samplewise convergence in time of the infinite feature series.
\end{proof}

\section{The full profile space and the injectivity threshold}
\label{sec:completion}

Theorems~\ref{lo:pressure} and~\ref{up:pressure}, together with the
row-count coupling \eqref{lo:poissontransfer}, prove
Theorem~\ref{main:pressure}. We now extend its functional to all overlap
distributions and pass to zero temperature.

\subsection{Completion of the functional}
Let $\zeta$ be a probability measure on $[0,1]$, and define
\begin{equation}\label{full:entropy}
 x_\zeta(s)=\zeta([0,s]),\qquad
 L_\zeta(s)=\int_s^1x_\zeta(u)\,\dd u,\qquad
 \mathcal E(\zeta)=\frac12\int_0^1
       \left(\frac1{L_\zeta(s)}-\frac1{1-s}\right)\dd s.
\end{equation}
Since $L_\zeta(s)\leq1-s$, the final integral is a nonnegative extended
Lebesgue integral, with $1/0=+\infty$.

There is a unique unit-diagonal cascade array law $\mathscr R_\zeta$
whose off-diagonal marginal is $\zeta$. To construct it, approximate
$\zeta$ weakly by finitely supported measures with support below one
and take a subsequential limit of their finite cascade laws. Exchangeability,
positive semidefiniteness and the Ghirlanda--Guerra identities pass to the
limit, initially for continuous cylinder tests and then for measurable
tests. Proposition~\ref{rpc:representation} identifies every subsequential
limit with the same law. The same argument shows that
\begin{equation}\label{full:arraycontinuity}
 \zeta_j\Rightarrow\zeta\quad\Longrightarrow\quad
 \mathscr R_{\zeta_j}\Rightarrow\mathscr R_\zeta.
\end{equation}

Given an array $R$ with this law, sample a centered Gaussian sequence
$(X^\ell)_{\ell\geq1}$ with conditional covariance $R$. Put
\begin{equation}\label{full:pattern}
 A_L=\frac1L\sum_{\ell=1}^Le^{-\beta\ind\{X^\ell>0\}},\qquad
 A_\infty=L^2\text{-}\lim_{L\to\infty}A_L,\qquad
 \mathcal J_\beta(\zeta)=\E\log A_\infty.
\end{equation}
Exchangeability gives this limit and
\[
 \E\abs{A_L-A_\infty}^2
 =\frac1L\left(\E e^{-2\beta\ind\{X^1>0\}}
       -\E e^{-\beta\ind\{X^1>0\}-\beta\ind\{X^2>0\}}\right)
 \leq\frac1L.
\]
In particular $e^{-\beta}\leq A_\infty\leq1$. Jensen's inequality,
first for the empirical average and then for its expectation, yields
\begin{equation}\label{full:patternbounds}
 -\beta/2\leq\mathcal J_\beta(\zeta)
 \leq\log\frac{1+e^{-\beta}}2.
\end{equation}
Set $\mathcal P_{\alpha,\beta}(\zeta)=
\mathcal E(\zeta)+\alpha\mathcal J_\beta(\zeta)$.

\begin{theorem}\label{full:completion}
For every finite $\alpha\geq0$ and $\beta>0$,
\begin{equation}\label{full:minimum}
 P_*(\alpha,\beta)
 =\min_{\zeta\in\mathcal P([0,1])}\mathcal P_{\alpha,\beta}(\zeta).
\end{equation}
Every minimizer has finite entropy and no atom at one.
\end{theorem}
\begin{proof}
We first show that $\mathcal J_\beta$ is continuous in $\zeta$.
The logarithm is $e^\beta$-Lipschitz on $[e^{-\beta},1]$, so
\[
 \abs{\E\log A_L-\mathcal J_\beta(\zeta)}\leq e^\beta L^{-1/2}
\]
uniformly in the overlap law. At fixed $L$, the Gaussian expectation
$\E\log A_L$ is continuous in the covariance matrix. The integrand's
discontinuities lie on coordinate hyperplanes, each of Gaussian measure
zero because the diagonal variances are one, including for singular
covariance matrices. Now use \eqref{full:arraycontinuity} and the uniform
approximation.

For a finite profile, $\zeta=\sum_{i=0}^k\pi_i\delta_{q_i}$ and
\[
 L_\zeta(s)=
 \begin{cases}
 \lambda_0,&0\leq s\leq q_0,\\
 \lambda_i+x_i(q_i-s),&q_{i-1}\leq s\leq q_i,\quad1\leq i\leq k,\\
 1-s,&q_k\leq s<1.
 \end{cases}
\]
Integrating \eqref{full:entropy} gives \eqref{eq:entropy}.
The marking recursion, including the independent residual Gaussians in
\eqref{rpc:residual}, gives
$\mathcal J_\beta(\zeta)=\mathcal J_\beta(q,x)$.

For $r<1$, let $\zeta_r$ be the image of $\zeta$ under $s\mapsto s\wedge r$.
Its distribution function equals $x_\zeta$ below $r$ and equals one above
$r$. Consequently $L_{\zeta_r}\geq L_\zeta$ and
$L_{\zeta_r}(s)\downarrow L_\zeta(s)$ as $r\uparrow1$. Monotone
convergence and continuity of the pattern term imply
\begin{equation}\label{full:capping}
 \mathcal E(\zeta_r)\uparrow\mathcal E(\zeta),\qquad
 \mathcal J_\beta(\zeta_r)\longrightarrow\mathcal J_\beta(\zeta).
\end{equation}
For fixed $r$, approximate $\zeta_r$ by finite measures on $[0,r]$.
The denominators in \eqref{full:entropy} are at least $1-r$ below $r$,
and the integrand is zero above $r$. Dominated convergence proves
continuity of the entropy on this capped domain. Combining the two
approximations gives finite profiles whose functional values approach
that of every finite-entropy $\zeta$. Infinite entropy cannot be canceled
by \eqref{full:patternbounds}. This proves equality of the infima.

Each map $\zeta\mapsto\mathcal E(\zeta_r)$ is continuous, since capping
preserves weak convergence. Thus \eqref{full:capping} makes
$\mathcal E$ lower semicontinuous. Compactness of $\mathcal P([0,1])$
and bounded continuity of $\mathcal J_\beta$ give a minimizer. The trial
$\delta_0$ has finite value, so every minimizer has finite entropy.
An atom of mass $a>0$ at one would give
$L_\zeta(s)\leq(1-a)(1-s)$ and hence infinite entropy.
\end{proof}

Finite entropy does not require support bounded away from one. For example,
the uniform distribution on $[0,1]$ satisfies
$L_\zeta(s)=(1-s^2)/2$ and $\mathcal E(\zeta)=(\log2)/2$.

\subsection{Relation to the Parisi equation}
Let $q$ be a nondecreasing quantile of $\zeta$, with essential endpoint
values. The quantile identity
\begin{equation}\label{full:quantile}
 L_\zeta(q(x))=1-xq(x)-\int_x^1q(y)\,\dd y
\end{equation}
identifies $L_\zeta(q(x))$ with the function $\lambda(x)$ in
\cite[equation (48)]{Maillard2025Injectivity}. For a smooth profile
with $q(1)<1$, a change of variables in \eqref{full:entropy} gives
\[
 \mathcal E(\zeta)=\frac12\log(1-q(1))
   +\frac{q(0)}{2\lambda(0)}
   +\frac12\int_0^1\frac{q'(x)}{\lambda(x)}\,\dd x.
\]
A jump at mass $x_i$ instead contributes
$\frac1{2x_i}\log(\lambda_{i-1}/\lambda_i)$, as in
\eqref{eq:entropy}.

For the pattern term, the transformation
$f\mapsto x^{-1}\log\E e^{xf(h+\sqrt tG)}$ is the nonlinear heat flow
with generator $(f_{hh}+xf_h^2)/2$. In the variance coordinate $s$, its
coefficient is $x_\zeta(s)$. On a fixed capped interval $[0,r]$ the
terminal convolution is Lipschitz in $h$. The flow preserves its
Lipschitz constant $L$. For two mass functions $a,b$, subtracting the
equations gives a linear parabolic equation for their difference, with
a bounded transport coefficient and forcing bounded by
$L^2\abs{a-b}/2$. The comparison principle therefore bounds the
supremum norm of the solution difference by
\[
 \frac{L^2}{2}\int_0^r\abs{a(s)-b(s)}\,\dd s.
\]
Approximation by step functions identifies the continuous cascade
functional with the Parisi equation of
\cite[equation (49)]{Maillard2025Injectivity} on its regular domain.
Equation~\eqref{full:capping} defines the extension at top overlap one
without subtracting divergent entropy terms.

\subsection{The energy-one crossing}
\begin{proof}[Proof of Theorem~\ref{main:threshold}]
For each matrix, convexity of $\beta\mapsto\log Z_{n,m}(\beta)$ and
$\log Z_{n,m}(0)=0$ imply that
$T_{n,m}(\beta)=-\log Z_{n,m}(\beta)/(\beta n)$ is nonincreasing in
$\beta>0$. We have $0\leq T_{n,m}(\beta)\leq m/(2n)$.
Theorem~\ref{main:pressure} therefore gives a nonincreasing function
\[
 T(\alpha,\beta)=-P_*(\alpha,\beta)/\beta\in[0,\alpha/2]
\]
and establishes existence of \eqref{eq:limiting-energy}.

By Theorem~\ref{full:completion},
\[
 T(\alpha,\beta)=\sup_{\mathcal E(\zeta)<\infty}
  \left\{-\frac{\mathcal E(\zeta)}\beta
                  -\frac\alpha\beta\mathcal J_\beta(\zeta)\right\}.
\]
The slopes of these affine functions of $\alpha$ belong to $[0,1/2]$,
by \eqref{full:patternbounds}. Also $\mathcal E\geq0$ and
$\mathcal E(\delta_0)=0$. Hence $T(0,\beta)=0$, and $T(\cdot,\beta)$
is convex, nondecreasing and $1/2$-Lipschitz. Its finite pointwise limit
$e$ has the same properties and satisfies $e(0)=0$.

For a sequence $m_n/n\to\alpha>1$, choose $1<a<\alpha<A$ with $A\geq2$
and apply Corollary~\ref{geo:bridge}. At fixed $\beta$, the lower inequality in
\eqref{eq:bridge} gives
\[
 \limsup_n\frac{h(m_n,n)}n\leq T(\alpha,\beta).
\]
The upper inequality gives the reverse bound for the liminf, up to
$\delta+C_{a,A,\delta}/\beta$. Sending $\beta\to\infty$ and then
$\delta\downarrow0$ proves
\begin{equation}\label{full:energy-limit}
 \frac{h(m_n,n)}n\longrightarrow e(\alpha),\qquad
 E_{n,m_n}\longrightarrow e(\alpha)\quad\text{in probability},
 \qquad \alpha>1.
\end{equation}
The probability convergence follows from \eqref{eq:count-concentration}.

At density $3/2$, the chamber count \eqref{eq:chamber-count} gives a
strictly negative direction with probability tending to one. Since
$0\leq E_{n,m}\leq m/(2n)$, the mean minimum tends to zero, and the
preceding limit yields $e(3/2)=0$. Nonnegativity and monotonicity give
$e(\alpha)=0$ for $0\leq\alpha\leq1$. For a sequence tending to any such
$\alpha$, row comparison bounds its mean minimum by that at
$\lfloor3n/2\rfloor$ rows. Its normalized minimum therefore tends to zero
in mean and in probability. Thus \eqref{full:energy-limit} holds for every
finite $\alpha\geq0$. The variables $E_{n,m_n}$ are uniformly bounded
along each ratio sequence, so the convergence is also in $L^1$.

Lemma~\ref{geo:anchors} gives $e(4)\leq1$. At density $16$, injectivity
holds with probability tending to one, and $E_{n,m}\geq1$ on that event.
Thus $e(16)\geq1$. Continuity gives a root $\alpha_\star\in[4,16]$.
Convexity and $e(0)=0$ imply that $e(\alpha)/\alpha$ is nondecreasing
on $(0,\infty)$. Consequently
\begin{equation}\label{full:strict-crossing}
 e(\alpha)\leq\alpha/\alpha_\star<1\quad(\alpha<\alpha_\star),
 \qquad
 e(\alpha)\geq\alpha/\alpha_\star>1\quad(\alpha>\alpha_\star).
\end{equation}
The root is therefore unique on $[0,\infty)$.

For $m_n/n\to\alpha\ne\alpha_\star$, set
$d=\abs{\alpha/\alpha_\star-1}>0$.
Equations~\eqref{full:energy-limit} and~\eqref{full:strict-crossing}
put $h(m_n,n)/n$ on the corresponding side of one at distance at least
$d/2$, for all sufficiently large $n$.  For positive $m_n$,
\eqref{eq:count-concentration} therefore bounds the wrong-phase
probability by
\[
 2\exp\left(-\frac{n^2d^2}{2m_n}\right).
\]
If $m_n=0$, the map is identically zero. Lemma~\ref{geo:criterion}
now proves both phases.
\end{proof}

\begin{corollary}\label{full:limsup-phases}
The noninjective conclusion holds if
$\limsup_n m_n/n<\alpha_\star$, and the injective conclusion holds if
$\liminf_n m_n/n>\alpha_\star$, including when $m_n/n\to\infty$.
\end{corollary}
\begin{proof}
Compare the row count with a fixed intermediate density and use monotonicity
of injectivity under adding rows.
\end{proof}

\section{Finite-family limits and the rational certificate}\label{app:trials}

\subsection{The finite-family limit}
\begin{proof}[Proof of Proposition~\ref{num:family}]
For sufficiently large $\beta$, use the admissible finite-temperature
profile
\[
 q_i(\beta)=q_i\quad(i<k),\qquad
 q_k(\beta)=1-\chi/\beta,\qquad x_i(\beta)=m_i/\beta.
\]
The standard Gaussian tail bounds give, locally uniformly in $h$,
\[
 \frac1\beta\log\bigl[
  \Phi_{\rm G}(-h\sqrt{\beta/\chi})
       +e^{-\beta}\Phi_{\rm G}(h\sqrt{\beta/\chi})\bigr]
       \longrightarrow-\psi_\chi(h),
\]
where $\Phi_{\rm G}$ is the standard normal distribution function.
Indeed, for $t\geq0$, its negative tail lies between
$c e^{-t^2/2}/(1+t)$ and $e^{-t^2/2}/2$ for a fixed $c>0$.
On each compact interval the resulting logarithmic error is
$O(\log\beta/\beta)$. All rescaled row functions take values in $[-1,0]$.
Couple the Gaussian increments as their variances converge. Bounded
convergence in each recursion, followed by the root expectation, gives
$\mathcal J_\beta/\beta\to J_0$; local uniform convergence is preserved
by the recursions, using bounded tails. Also $\beta\lambda_i\to\ell_i$,
and the entropy formula gives $\mathcal E/\beta\to E_0$.
These arguments include $k=0$, where the root variance tends to one.

For every sufficiently large $\beta$,
$P_*(\alpha,\beta)\leq\mathcal P_{\alpha,\beta}(q(\beta),x(\beta))$.
Taking $-1/\beta$ times this inequality and passing to the limit proves
$B_\alpha\leq e(\alpha)$. The Gaussian recursion has $J_0<0$:
the terminal cost is positive on a set of positive Gaussian measure,
and this strict inequality persists through every recursion.
Thus $B_C=1$, so $e(C)\geq1$ and the unique-crossing theorem gives
$\alpha_\star\leq C$.
\end{proof}

\subsection{The one-step certificate}
\begin{proof}[Proof of Corollary~\ref{num:seven}]
Take $k=1$, $q=q_0=39/50$, $\chi=1/20$ and $m=m_1=16/5$.
Write $s=\sqrt{1-q}$, $d=\chi+ms^2$, $a=\sqrt{2\chi}$,
$\mu=\chi h/d$ and $\tau=\sqrt{\chi s^2/d}$.
Completing the square in the interval $0<h+sG<a$ gives
\begin{align*}
 A(h)&=\E e^{-m\psi_\chi(h+sG)}\\
 &=\Phi_{\rm G}(-h/s)
   +\sqrt{\chi/d}\,e^{-mh^2/(2d)}
     \left[\Phi_{\rm G}((a-\mu)/\tau)-\Phi_{\rm G}(-\mu/\tau)\right]
   +e^{-m}\Phi_{\rm G}((h-a)/s).
\end{align*}
Set $F(z)=-\log A(\sqrt q\,z)/m$. Since $\psi_\chi$ is nondecreasing,
$F$ is nondecreasing and $0\leq F\leq1$. With $\varphi_{\rm G}$ the
standard normal density, the trial value is
\[
 B_7=7\int_\R F(z)\varphi_{\rm G}(z)\,\dd z
       -\frac{195}{377}-\frac5{32}\log\frac{377}{25}.
\]
For $z_j=-4+j/50$, $0\leq j\leq400$, and
$p_j=\Phi_{\rm G}(z_{j+1})-\Phi_{\rm G}(z_j)$, monotonicity gives
\begin{align*}
 \sum_{j=0}^{399}F(z_j)p_j+F(4)\Phi_{\rm G}(-4)
 &\leq\int_\R F\varphi_{\rm G}\\
 &\leq\sum_{j=0}^{399}F(z_{j+1})p_j
       +(F(-4)+1)\Phi_{\rm G}(-4).
\end{align*}
Outward rational evaluation of these sums gives
\begin{equation}\label{num:certificate}
 \frac{1060394076597313061}{10^{18}}
     \leq B_7\leq
 \frac{1102089488307234572}{10^{18}}.
\end{equation}
For completeness, the ancillary checker uses integer-square-root brackets,
Machin's formula for $\pi$, the alternating series for the error function
$\operatorname{erf}(x)=2\pi^{-1/2}\int_0^x e^{-t^2}\,\dd t$
after its terms decrease, and range-reduced exponential and logarithm
series with geometric tail bounds. Gaussian arguments outside $[-4,4]$
are enclosed by symmetry and the Mills bound
$\Phi_{\rm G}(-4)\leq e^{-8}/(4\sqrt{2\pi})$.
Every arithmetic operation rounds its lower endpoint down and its upper
endpoint up to multiples of $10^{-18}$. The checker \path{anc/verify_one_step.py} and its rational parameter file
\path{anc/one_step_parameters.json} are included in the source package;
Section~\ref{sec:numerics} gives the execution command. Thus
\eqref{num:certificate} includes arithmetic, Gaussian-tail and integration
errors. Its lower endpoint exceeds one. Proposition~\ref{num:family}
gives $e(7)>1$, hence $\alpha_\star<7$.
\end{proof}

\section{AI assistance}
The large language model GPT-6.1 Sol was used to assist with editing, proof
checking, and verification. The author takes full responsibility for the
content of this paper.
\bibliographystyle{amsplain}
\bibliography{references}
\end{document}